\documentclass[a4paper]{amsart}
\usepackage{graphicx}
\usepackage{amsmath,amsfonts}
\usepackage{amsthm,amssymb,latexsym}
\usepackage[active]{srcltx}

\newtheorem{theorem}{Theorem}[section]
\newtheorem{corollary}[theorem]{Corollary}
\newtheorem{lemma}[theorem]{Lemma}

\newtheorem*{fact*}{Fact}
\newtheorem{proposition}[theorem]{Proposition}
\theoremstyle{definition}

\newtheorem*{claim}{Claim}
\newtheorem{remark}[theorem]{Remark}
\numberwithin{equation}{section}
\newtheorem*{algorithm}{Algorithm SVS}
\newcommand{\N}{\mathbb N}
\newcommand{\Z}{\mathbb Z}

\newcommand{\F}{\mathbb F}

\newcommand{\R}{\mathbb R}

\newcommand{\M}{{\sf M}}
\newcommand{\FF}{{\sf F}}

\newcommand{\ds}{\displaystyle}

\newcommand{\bfs}{\boldsymbol}

\newcommand{\fq}{\F_{\hskip-0.7mm q}}

\newcommand{\cfq}{\overline{\F}_{\hskip-0.7mm q}}
\def\ifm#1#2{\relax \ifmmode#1\else#2\fi}

\newcommand{\stirling}[2]{\genfrac{[}{]}{0pt}{}{#1}{#2}}

\newcommand{\klk}    {\ifm {,\ldots,} {$,\ldots,$}}

\newcommand{\plp}    {\ifm {+\cdots+} {$+\ldots+$}}

\begin{document}

\title[Computation of rational points on hypersurfaces]{On the
computation of rational points of a hypersurface over a finite
field}
\author[G. Matera]{Guillermo Matera${}^{1,2}$}%

\author[M. P\'erez]{Mariana P\'erez${}^1$}

\author[M. Privitelli]{Melina Privitelli${}^3$}

\address{${}^{1}$Instituto del Desarrollo Humano,
Universidad Nacional de Gene\-ral Sarmiento, J.M. Guti\'errez 1150
(B1613GSX) Los Polvorines, Buenos Aires, Argentina}
\email{\{gmatera,\,vperez\}@ungs.edu.ar}
\address{${}^{2}$ National Council of Science and Technology (CONICET),
Ar\-gentina}
\address{${}^{3}$Instituto de Ciencias,
Universidad Nacional de Gene\-ral Sarmiento, J.M. Guti\'errez 1150
(B1613GSX) Los Polvorines, Buenos Aires, Argentina}
\email{mprivite@ungs.edu.ar}

\thanks{The authors were partially supported by the grants PIP CONICET
11220130100598, PIO CONICET-UNGS 14420140100027 and UNGS 30/3084}%
\date{\today}
\maketitle

\begin{abstract}
We design and analyze an algorithm for computing rational points of
hypersurfaces defined over a finite field based on searches on
``vertical strips'', namely searches on parallel lines in a given
direction. 
Our results show that, on average, less than two searches suffice to
obtain a rational point. We also analyze the probability
distribution of outputs, using the notion of Shannon entropy, and
prove that the algorithm is somewhat close to any ``ideal''
equidistributed algorithm.
\end{abstract}
%
%
\section{Introduction}\label{section: intro}
Let $\fq$ be the finite field of $q$ elements, $X_1,\dots,X_r$
indeterminates over $\fq$ and $\fq[X_1,\dots,X_r]$ the ring of
polynomials in $X_1,\dots,X_r$ with coefficients in $\fq$. Let
$\mathcal{F}_{r,d}:=\{F \in \fq[X_1,\dots, X_r]: \deg(F)\leq d\}$.
Suppose that $r\ge 2$ and $d\ge 2$, and let $F$ be an element of
$\mathcal{F}_{r,d}$. In this paper we address the problem of finding
an $\fq$--rational zero of $F$, namely a point $\bfs{x}\in\fq^r$
with $F(\bfs{x})=0$.

It is well--known that the elements of $\mathcal{F}_{r,d}$ have
$q^{r-1}$ zeros in $\fq^r$ on average. More precisely, we have the
following result (see, e.g., \cite[Theorem 6.16]{LiNi83}):
\begin{equation}\label{eq: average number zeros}
\frac{1}{|\mathcal{F}_{r,d}|}\sum_{F\in\mathcal{F}_{r,d}}
N(F)=q^{r-1},
\end{equation}
where $N(F):=|\{\bfs x\in\fq^{r}:F(\bfs x)=0\}|$. 
%
This suggests a strategy to find an $\fq$--rational zero of a given
$F\in\mathcal{F}_{r,d}$. Since the expected number of zeros of $F$
is equal to the cardinality of $\fq^{r-1}$, given $\bfs
a_1\in\fq^{r-1}$, one may try to find a zero of $F$ having $\bfs
a_1$ as its first $r-1$ coordinates. If the polynomial $F(\bfs
a_1,X_r)$ has no zeros in $\fq$, then a further element $\bfs
a_2\in\fq^{r-1}$ can be picked up to see whether $F(\bfs a_2,X_r)$
has a zero in $\fq$. The algorithm proceeds in this way until a zero
of $F$ in $\fq^r$ is obtained.

Following the terminology of \cite{GaShSi03}, which considers the
case $r=2$, each set $\{\bfs a_i\}\times \fq$ is called a ``vertical
strip''. Therefore, our algorithm, which extends the one of
\cite{GaShSi03} to $r$--variate polynomials, is called ``Search on
Vertical Strips'' (SVS for short), and is described as follows.
\begin{algorithm} ${}$

Input: a polynomial $F\in\mathcal{F}_{r,d}$.

Output: either a zero $\bfs x\in\fq^r$ of $F$, or ``failure''.

Set $i:=1$ and $f:=1$

While $1\leq i\leq q^{r-1}$ and $f=1$ do
\begin{enumerate}
  \item[] Choose at random $\bfs a_i\in\fq^{r-1}\setminus\{\bfs a_1\klk \bfs a_{i-1}\}$
  \item[] Compute $f:=\gcd(F(\bfs a_i,X_r),X_r^q-X_r)$
  \item[] If $f=0$, then choose $x_{r,i}\in\fq$ at random
  \item[] If $f\notin\{0,1\}$, then compute a root $x_{r,i}\in\fq$ of
  $f$
  \item[] $i:=i+1$
\end{enumerate}

End While

If $f\not=1$ return $(\bfs a_i,x_{r,i})$, else return ``failure''.
\end{algorithm}

Ignoring the cost of random generation of elements of $\fq^{r-1}$,
at the $i$th step of the main loop we compute the vector of
coefficients of the polynomial $F(\bfs a_i,X_r)$. Since an element
of $\mathcal{F}_{r,d}$ has $D:=\binom{d+r}{r}$ coefficients, the
number of arithmetic operations in $\fq$ required to compute such a
vector is $\mathcal{O}^\sim(D)$, where the notation
$\mathcal{O}^\sim$ ignores logarithmic factors. Throughout this
paper, all asymptotic estimates are valid for fixed $d$ and $r$, and
$q$ growing to infinity. Then the gcd $f$ is computed, and a root of
$f$ in $\fq$ is determined, provided that $f\not=1$. This can be
done with $\mathcal{O}^\sim(d\,\log_2q)$ arithmetic operations in
$\fq$ (see, e.g., \cite[Corollary 14.16]{GaGe99}). As a consequence,
for a choice $\underline{\bfs a}:=(\bfs a_1\klk \bfs a_{q^{r-1}})$
for the vertical strips to be considered, the whole procedure
requires $\mathcal{O}^\sim\big(C_{\underline{\bfs a}}(F)\cdot
(D+d\log_2q)\big)$ arithmetic operations in $\fq$, where
$C_{\underline{\bfs a}}(F)$ is the least value of $i$ for which
$F(\bfs a_i,X_r)$ has a zero in $\fq$.

This paper is devoted to analyze the SVS algorithm from a
probabilistic point of view. As its behavior is essentially
determined by the number of vertical strips which must be
considered, we analyze, for a given $s\ge 1$, the probability
distribution of the number of searches performed by the algorithm.
For this purpose, we consider the set $\FF$ of all possible choices
of vertical strips and the random variable $C_{r,d}:\FF\times
\mathcal{F}_{r,d}\mapsto \N$ which counts the number of vertical
strips that are searched. We prove that the probability that $s$
vertical strips are searched, for ``moderate'' values of $s$,
satisfies the estimate
\begin{equation}\label{eq: estimate prob s searches intro}
P[C_{r,d}=s]=(1-\mu_d)^{s-1}\mu_d+\mathcal{O}(q^{-1/2}),
\end{equation}
where $\mu_d:=\sum_{j=1}^d(-1)^{j-1}/j!$. Observe that $\mu_d\approx
1-e^{-1}= 0.6321\ldots$ for large $d$, where $e$ denotes the basis
of the natural logarithm. We remark that the quantity $\mu_d$ arises
also in connection with a classical combinatorial notion over finite
fields, that of the value set of univariate polynomials (cf.
\cite{LiNi83}, \cite{MuPa13}). For a polynomial $f\in\fq[T]$, denote
by $\mathcal{V}(f):=|\{f(c):c\in\fq\}|$ the cardinality of the value
set of $f$. In \cite{BiSD59}, Birch and Swinnerton--Dyer established
the following classical result: if $f\in\fq[T]$ is a generic
polynomial of degree $d$, then
$\mathcal{V}(f)=\mu_d\,q+\mathcal{O}(1)$.

The estimate \eqref{eq: estimate prob s searches intro} relies on
the analysis of the behavior of the SVS algorithm for a fixed choice
$\bfs{a}_1,\dots,\bfs{a}_s\in\fq^{r-1}$ for the first $s$ vertical
strips. 
It turns out that the probability that the $s$ vertical strips
under consideration are searched is essentially that of the
right--hand side of \eqref{eq: estimate prob s searches intro}. As a
side note, this may be considered as a ``realistic'' version of the
SVS algorithm in the sense of \cite{Bach91}. As the author states,
``when a randomized algorithm is implemented, one always uses a
sequence whose later values come from earlier ones in a
deterministic fashion. This invalidates the assumption of
independence and might cause one to regard results about
probabilistic algorithms with suspicion.'' Our results show that the
probabilistic behavior of the SVS algorithm is not essentially
altered when a fixed choice of vertical strips is considered.

As a consequence of \eqref{eq: estimate prob s searches intro} we
obtain an upper bound on the average--case complexity $E[X]$ of the
SVS algorithm, where $X:\FF\times \mathcal{F}_{r,d}\to\N$ is the
random variable that counts the number of arithmetic operations in
$\fq$ performed for a given choice of vertical strips on a given
input. We prove that
\begin{equation}\label{eq: average case complex intro}
E[X]\le \frac{1}{\mu_d}\tau(d,r, q)+\mathcal{O}(q^{-1/2}),
\end{equation}
where $\tau(d,r,q):=\mathcal{O}^\sim(D+d\log_2 q)$ is the cost of a
search in a single vertical strip. In other words, on average at
most $1/\mu_d\approx 1.58$ vertical strips must be searched to
obtain a rational zero of the polynomial under consideration.
Simulations we run suggest that the upper bound \eqref{eq: average
case complex intro} is close to optimal. We observe that the
probabilistic algorithms of \cite{GaShSi03} (for $r=2$) and
\cite{CaMa06a} and \cite{Matera10} (for general $r$) propose $d$
searches in order to achieve a probability of success greater than
1/2. Our result suggests that these analyses are somewhat
pessimistic.

On the other hand, it must be said that the result of
\cite{GaShSi03} holds for any bivariate polynomial, while that of
\cite{CaMa06a} is valid for any absolutely irreducible $r$--variate
polynomial. If the polynomials under consideration are produced by
some complicated process, it might be argued that our results do not
contribute to the analysis of the cost of the corresponding
algorithm to search for $\fq$--rational zeros. Nevertheless, a
crucial aspect of our approach is that we express the probability
$P[C_{r,d}=s]$ of \eqref{eq: estimate prob s searches intro}, and
thus the average--case complexity $E[X]$ of \eqref{eq: average case
complex intro}, in terms of the average cardinality of the value set
of certain families of univariate polynomials related to the set of
input polynomials under consideration. We believe that this
technique can be extended to deal with (linear or nonlinear)
families of polynomials of $\mathcal{F}_{r,d}$, provided that the
asymptotic behavior of the average cardinality of the corresponding
families of univariate polynomials is known (see \cite{CeMaPePr14},
\cite{MaPePr14} and \cite{MaPePr16} for results in connection with
this matter).

Another critical aspect to analyze is the distribution of outputs.
Given $F\in\mathcal{F}_{r,d}$, the SVS algorithm outputs an
$\fq$--rational zero of $F$, which is determined by certain random
choices made during its execution. As a consequence, it is relevant
to have insight on the probability distribution of outputs. For an
``ideal'' algorithm (from the point of view of distribution of
outputs), outputs should be equidistributed. For this reason, in
\cite{GaShSi03} the basic SVS strategy for bivariate polynomials
over $\fq$ is modified so that all $\fq$--rational zeros of the
input polynomial are equally probable outputs. Such a modification
can be also be applied to our algorithm.

Nevertheless, as this modification implies a certain slowdown, we
shall pursue a different course of action, analyzing the average
distribution of outputs by means of the concept of Shannon entropy.
If the output for an input polynomial $F$ tends to be concentrated
on a few $\fq$--rational zeros of $F$, then the ``amount of
information'' that we obtain might be said to be ``small''. On the
other hand, if all the $\fq$--rational zeros of $F$ are equally
probable outputs, then the amount of information provided by the
algorithm is considered to be larger. Following \cite{BePa11} (see
also \cite{BeLe12}), we define a Shannon entropy $H_F$ associated to
an input $F\in\mathcal{F}_{r,d}$ of the SVS algorithm, which
measures how ``concentrated'' are the corresponding outputs. Then we
analyze the average entropy $H$ when $F$ runs through all the
elements of $\mathcal{F}_{r,d}$.

For an ``ideal'' algorithm for computing $\fq$--rational zeros of
elements of $\mathcal{F}_{r,d}$ and $F\in\mathcal{F}_{r,d}$, it is
easy to see that $H^{\rm ideal}_F=\log N(F)$, where $\log$ denotes
the natural logarithm. It follows that
%
$$H^{\rm ideal}\le \log(q^{r-1})$$
%
(see \eqref{eq: bound entropy ideal algorithm}). Our main result
concerning the distribution of outputs asserts that
\begin{equation}\label{eq: average entropy intro}
H\ge \frac{1}{2\mu_d}\log(q^{r-1}) (1+\mathcal{O}(q^{-1})).
\end{equation}
Since ${1}/{2\mu_d}\approx 0.79$ for large $d$, we may paraphrase
\eqref{eq: average entropy intro} as saying that the SVS algorithm
is at least $79$ per cent as good as any ``ideal'' algorithm, from
the point of view of the distribution of the outputs.

The proof of \eqref{eq: average entropy intro} relies on an analysis
of the expected number of vertical strips of the elements of
$\mathcal{F}_{r,d}$ which may be of independent interested. Denote
by $NS(r,d)$ the average number of vertical strips with
$\fq$--rational zeros of $F$, when $F$ runs through all the elements
of $\mathcal{F}_{r,d}$. We prove that
\begin{equation}\label{eq: average number VS intro}
NS(r,d)=\mu_d\,q^{r-1}+\mathcal{O}(q^{r-2}).
\end{equation}
We also estimate the variance of the number of vertical strips with
$\fq$--rational zeros.

The paper is organized as follows. Section \ref{section: analysis
C=1 and C=2} is devoted to the analyses of the probability that one
or two vertical strips are searched. In Section \ref{section:
analysis fixed ordered VS} we estimate the expected number of
vertical strips to be searched for a given choice of $s\ge 3$
vertical strips. We express the probability that $s$ vertical strips
are searched in terms of average cardinalities of value sets and
apply estimates for the latter in order to establish an explicit
estimate of the former. In Section \ref{section: analysis of SVS
algorithm} we apply the results of Sections \ref{section: analysis
C=1 and C=2} and \ref{section: analysis fixed ordered VS} to
establish \eqref{eq: estimate prob s searches intro} and \eqref{eq:
average case complex intro}. Section \ref{section: entropy} is
concerned with the probability distribution of outputs. In
Subsection \ref{subsec: average number vertical strips} we establish
\eqref{eq: average number VS intro} and an estimate of the
corresponding variance. In Subsection \ref{subsec: bound entropy
with replacement} we apply these estimates to prove \eqref{eq:
average entropy intro}. Finally, in Section \ref{section:
simulations} we exhibit a few simulations aimed at confirming the
asymptotic results \eqref{eq: estimate prob s searches intro} and
\eqref{eq: average case complex intro}.
%
%
\section{Probability of success in the first two searches}
\label{section: analysis C=1 and C=2}
We start discussing how frequently one or two searches on vertical
strips suffice to find a zero of the input polynomial. As it will
become evident, this will happen in most cases. Therefore, accurate
estimates on the probability of these two cases is critical for an
accurate description of the behavior of the algorithm.
%
%
\subsection{Probability of success in the first search}
%
For integers $r\ge 2$ and $d\ge 2$, we shall estimate the
probability that the SVS algorithm, on input an element of
$\mathcal{F}_{r,d}:=\{F \in \fq[X_1,\dots, X_r]: \deg(F)\leq d\}$,
finds a root of it in the first vertical strip. As $r$ and $d$ are
fixed, we shall drop the indices $r$ and $d$ from the notations.

Each possible choice for the first vertical strip is determined by
an element of $\fq^{r-1}$. As a consequence, we may represent the
situation by means of the random variable
$C_1:=C_{1,r,d}:\fq^{r-1}\times \mathcal{F}_{r,d}\to\{1,\infty\}$
defined in the following way:
$$
C_1(\bfs a,F):=\left\{
\begin{array}{cl}
1&\textrm{if }F(\bfs a,X_r)\textrm{ has an }\fq\textrm{--rational
zero},\\ \!\!\infty&\textrm{otherwise}.
\end{array}\right.
$$
We consider the set $\fq^{r-1}\times \mathcal{F}_{r,d}$ endowed with
the uniform probability $P_1:=P_{1,r,d}$ and study the probability
of the set $\{C_1=1\}$. The next result provides an exact formula
for this probability.
\begin{theorem}\label{th: prob C=1}
For $q>d$, we have the identity
\begin{align*}
P_1[C_1=1]=\sum_{j=1}^d(-1)^{j-1}\binom{q}{j}q^{-j}
+(-1)^d\binom{q-1}{d}q^{-d-1}.
\end{align*}
\end{theorem}
\begin{proof}
For any $F\in\mathcal{F}_{r,d}$, we denote by $VS(F)$ the set of
vertical strips where $F$ has an $\fq$--rational zero and by $NS(F)$
its cardinality, that is,
$$VS(F):=\{\bfs a\in\fq^{r-1}:(\exists\, x_r\in\fq)\ F(\bfs
a,x_r)=0\},\quad NS(F):=|VS(F)|.$$
It is easy to see that
$\{C_1=1\}=\bigcup_{F\in\mathcal{F}_{r,d}}VS(F)\times\{F\}.$
Since this is a union of disjoint subsets of $\fq^{r-1}\times
\mathcal{F}_{r,d}$, it follows that
\begin{equation}\label{eq: prob C=1 in terms of number vert strips}
P_1[C_1=1]=\frac{1}{q^{r-1}|\mathcal{F}_{r,d}|}
\sum_{F\in\mathcal{F}_{r,d}}NS(F).
\end{equation}

Fix $F\in\mathcal{F}_{r,d}$. Observe that
$$VS(F)=\bigcup_{x\in\fq}\{\bfs a\in\fq^{r-1}:F(\bfs a,x)=0\}.$$
As a consequence, by the inclusion--exclusion principle we obtain
\begin{align*}
NS(F)&=\Bigg|\bigcup_{x\in\fq}\{\bfs a\in\fq^{r-1}:F(\bfs a,x)=0\}
\Bigg|\\
&=\sum_{j=1}^q(-1)^{j-1}\sum_{\mathcal{X}_j\subset\fq} \big|\{\bfs
a\in\fq^{r-1}:(\forall x\in\mathcal{X}_j)\,F(\bfs a,x)=0\}\big|,
\end{align*}
where $\mathcal{X}_j$ runs through all the subsets of $\fq$ of
cardinality $j$. We conclude that
$$\sum_{F\in\mathcal{F}_{r,d}}NS(F)=\sum_{F\in
\mathcal{F}_{r,d}}\sum_{j=1}^q(-1)^{j-1}\sum_{\mathcal{X}_j\subset\fq}
\big|\{\bfs a\in\fq^{r-1}:(\forall x\in\mathcal{X}_j)\,F(\bfs
a,x)=0\}\big|.$$

For any $j$ with $1\le j\le q$, we denote
$$\mathcal{N}_j:=\frac{1}{q^{r-1}|\mathcal{F}_{r,d}|}\sum_{F\in
\mathcal{F}_{r,d}}\sum_{\mathcal{X}_j\subset\fq} \big|\{\bfs
a\in\fq^{r-1}:(\forall x\in\mathcal{X}_j)\,F(\bfs a,x)=0\}\big|,$$
where $\mathcal{X}_j$ runs through all the subsets of $\fq$ of
cardinality $j$. If $j\le d$ and $\bfs a$ is fixed, then the
equalities $F(\bfs a,x)=0$ ($x\in\mathcal{X}_j$) are $j$
linearly--independent conditions on the coefficients of $F$ in the
$\fq$--vector space $\mathcal{F}_{r,d}$. It follows that
\begin{align}
\mathcal{N}_j&=\frac{1}{q^{r-1}|\mathcal{F}_{r,d}|}\sum_{\mathcal{X}_j\subset\fq}
\sum_{\bfs a\in \fq^{r-1}} \big|\{F\in\mathcal{F}_{r,d}:(\forall
x\in\mathcal{X}_j)\,F(\bfs a,x)=0\}\big|\nonumber \\
&= \frac{1}{q^{r-1+\dim
\mathcal{F}_{r,d}}}\sum_{\mathcal{X}_j\subset\fq} \sum_{\bfs a\in
\fq^{r-1}}q^{\dim \mathcal{F}_{r,d}-j}=\binom{q}{j}q^{-j}.
\label{eq: aux proof th prob C=1 - 1}
\end{align}
On the other hand, if $j>d$, then $F(\bfs a,x)=0$ for every $x\in
\mathcal{X}_j$ if and only if $F(\bfs a,X_r)=0$. The condition
$F(\bfs a,X_r)=0$ is expressed by means of $d+1$
linearly--independent linear equations on the coefficients of $F$ in
$\mathcal{F}_{r,d}$. We conclude that
\begin{align}
\mathcal{N}_j= \frac{1}{q^{r-1+\dim
\mathcal{F}_{r,d}}}\sum_{\mathcal{X}_j\subset\fq} \sum_{\bfs a\in
\fq^{r-1}}q^{\dim \mathcal{F}_{r,d}-(d+1)}=\binom{q}{j}q^{-d-1}.
\label{eq: aux proof th prob C=1 - 2}
\end{align}
Combining \eqref{eq: aux proof th prob C=1 - 1} and \eqref{eq: aux
proof th prob C=1 - 2} we obtain
$$
P_1[C_1=1]=\sum_{j=1}^q(-1)^{j-1}\mathcal{N}_j=\sum_{j=1}^d(-1)^{j-1}\binom{q}{j}q^{-j}+
\sum_{j=d+1}^q(-1)^{j-1} \binom{q}{j}q^{-d-1}.$$
Finally, since
\begin{equation}\label{eq: identities combinatorial numbers}
\sum_{j=d+1}^q(-1)^{j-1} \binom{q}{j}=
\sum_{j=0}^d(-1)^j\binom{q}{j}=(-1)^d\binom{q-1}{d}
\end{equation}
(see, e.g., \cite[(5.16)]{GrKnPa94}), we readily deduce the
statement of the theorem.
\end{proof}

Next we discuss the asymptotic behavior of the probability
$P_1[C_1=1]$. Fix $d\ge 2$. From Theorem \ref{th: prob C=1} it can
be seen that
$$P_1[C_1=1]=\mu_d+\mathcal{O}(q^{-1}),\quad \mu_d:=
\sum_{j=1}^d\frac{(-1)^{j-1}}{j!}.$$
To show this, given positive integers $k,j$ with $k\le j$, we shall
denote by $\stirling{j}{k}$ the unsigned Stirling number of the
first kind, namely the number of permutations of $j$ elements with
$k$ disjoint cycles. The following properties of the Stirling
numbers are well--known (see, e.g., \cite[\S A.8]{FlSe08}):
$$\stirling{j}{j}=1,\ \ \stirling{j}{{j-1}}=\binom{j}{2},\ \
\sum_{k=0}^j\stirling{j}{k}=j!.$$
%
%
We shall also use the following well--known identity (see, e.g.,
\cite[(6.13)]{GrKnPa94}):
\begin{equation}\label{eq: binomial in terms of stirling numbers}
\binom {q}{j}=\sum_{k=0}^j\frac{(-1)^{j-k}}{j!}\stirling{j}{k}q^k.
\end{equation}
According to Theorem \ref{th: prob C=1} and \eqref{eq: binomial in
terms of stirling numbers}, we have
\begin{align*}
P_1[C_1=1]&= \sum_{j=1}^d(-1)^{j-1}\sum_{k=0}^j
\frac{(-1)^{j-k}}{j!}\stirling{j}{k}q^{k-j}+(-1)^d\binom{q-1}{d}q^{-d-1}
\\
&= \sum_{j=1}^d \frac{(-1)^{j-1}}{j!}\stirling{j}{j}+ \sum_{j=1}^d
\frac{(-1)^j}{j!}\stirling{j}{j-1}q^{-1}\\&\quad+
\sum_{j=1}^d\sum_{k=0}^{j-2}
\frac{(-1)^{k-1}}{j!}\stirling{j}{k}q^{k-j}+(-1)^d
\binom{q-1}{d}q^{-d-1}.
\end{align*}
It follows that
$$
P_1[C_1=1]=\mu_d+\frac{1}{q}\sum_{j=1}^d
\frac{(-1)^j}{j!}\binom{j}{2}- \sum_{j=1}^d\sum_{k=0}^{j-2}
\frac{(-1)^k}{j!}\stirling{j}{k}q^{k-j} +\frac{(-1)^d}{q^{d+1}}
\binom{q-1}{d}.
$$
As a consequence, for $d>2$ we obtain
\begin{align*}
\left|P_1[C_1=1]-\mu_d\right|&\le \frac{1}{q}\Bigg|\sum_{j=1}^d
\frac{(-1)^j}{j!}\binom{j}{2}\Bigg|+ \sum_{j=1}^d\sum_{k=0}^{j-2}
\frac{1}{j!}\stirling{j}{k}\frac{1}{q^2} +\frac{1}{q^{d+1}}
\binom{q-1}{d}\\
&\le \frac{1}{4q}+ \frac{d}{q^2} +\frac{1}{2q}.
\end{align*}
For $d=2$, this inequality is obtained by a direct calculation. We
have therefore the following result.
\begin{corollary}\label{coro: prob C=1 - asymptotic}
For $q>d$,
$$\big|P_1[C_1=1]-\mu_d\big|\le \frac{2}{q}.$$
\end{corollary}

As $d$ tends to infinity, the number $P_1[C_1=1]$ tends to
$1-e^{-1}=0.6321\ldots$, where $e$ denotes the basis of the natural
logarithm. This explains the numerical results in the first row of
the tables of the simulations of Section \ref{section:
simulations}.

It is worth remarking that the quantity $P_1[C_1=1]$ is closely
connected with the probability that a univariate polynomial of
degree at most $d$ has $\fq$--rational roots. More precisely,
consider the set $\mathcal{F}_{1,d}$ of univariate polynomials of
degree at most $d$ with coefficients in $\fq$, endowed with the
uniform probability $p_{1,d}$, and the random variable
$N_{1,d}:\mathcal{F}_{1,d}\to\Z_{\ge 0}$ which counts the number of
$\fq$--rational zeros, namely
$$
N_{1,d}(f):=|\{x\in \fq:\ f(x)=0\}|.
$$
%
The random variable $N_{1,d}$ has been implicitly studied in the
literature (see, e.g., \cite[\S 2]{Cohen73} or \cite[Theorem
3]{Knkn90}). It can be proved that, for $q>d$,
$$
p_{1,d}[N_{1,d}>0]=P_1[C_1=1].
$$
%
%
\subsection{Probability of success in the second search}
Next we analyze the probability that the SVS algorithm performs
exactly two searches.

Each possible choice for the first two vertical strips is determined
by an element $\underline{\bfs a}:=(\bfs a_1,\bfs
a_2)\in\fq^{r-1}\times\fq^{r-1}$ with $\bfs a_1\not=\bfs a_2$.
Therefore, we denote by $\FF_2$ the set of all such possible choices
and by $N_2$ its cardinality, that is,
$$\FF_2:=\{\underline{\bfs
a}:=(\bfs a_1,\bfs a_2)\in\fq^{r-1}\times\fq^{r-1}:\bfs a_1\not=\bfs
a_2\},\quad N_2=|\FF_2|=q^{r-1}(q^{r-1}-1).$$
We shall study the random variable $C_2:=C_{2,r,d}:\FF_2\times
\mathcal{F}_{r,d}\to\{1,2,\infty\}$ defined as
$$C_2(\underline{\bfs
a},F):=\left\{\begin{array}{cl} 1&\textrm{if }N_{1,d}(F(\bfs
a_1,X_r))>0,\\[0.25ex]
2&\textrm{if }N_{1,d}(F(\bfs a_1,X_r))=0 \textrm{ and
}N_{1,d}(F(\bfs
a_2,X_r))>0,\\[0.25ex]
\infty&\textrm{otherwise}.\\
\end{array}\right.$$
We consider the set $\FF_2\times\mathcal{F}_{r,d}$ endowed with the
uniform probability $P_2:=P_{2,r,d}$. We aim to determine the
probability $P_2[C_2=2]$.

This probability will be expressed in terms of probabilities
concerning the random variables $C_{\underline{\bfs
a}}:=C_{\underline{\bfs a},r,d}:\mathcal{F}_{r,d}\to\{1,2,\infty\}$
which count the number of searches that are performed on the
vertical strips defined by $\underline{\bfs a}:=(\bfs a_1,\bfs
a_2)\in\FF_2$ until an $\fq$--rational zero is obtained,
$C_{\underline{\bfs a}}(F)=\infty$ meaning that $F$ does not have
$\fq$--rational zeros on these two vertical strips. For this
purpose, the set $\mathcal{F}_{r,d}$ is considered to be endowed
with the uniform probability $p_{r,d}$. The relation between these
random variables and $P_2[C_2=2]$ is expressed in the following
lemma.
\begin{lemma}\label{lemma: p_2 in terms of C_r,d,a}
We have
\begin{align*}
P_2[C_2=2]&=\frac{1}{N_2} \sum_{\underline{\bfs a}\in\FF_2}
p_{r,d}[C_{\underline{\bfs a}}=2].
\end{align*}
\end{lemma}
\begin{proof}
Observe that
$$\{C_2=2\}=\bigcup_{\underline{\bfs a}\in\FF_2}
\{\underline{\bfs
a}\}\times\{F\in\mathcal{F}_{r,d}:C_{\underline{\bfs a}}(F)=2\}.$$
Since this is union of disjoint sets, we conclude that
\begin{align*}
P_2[C_2=2]&=\frac{1}{N_2} \sum_{\underline{\bfs a}\in\FF_2}
\frac{\big|\{F\in\mathcal{F}_{r,d}:C_{\underline{\bfs
a}}(F)=2\}\big|}{|\mathcal{F}_{r,d}|} =
\frac{1}{N_2}\sum_{\underline{\bfs a}\in\FF_2}
p_{r,d}[C_{\underline{\bfs a}}=2],
\end{align*}
which proves the lemma.
\end{proof}

Next we estimate the probability $p_{r,d}[C_{\underline{\bfs a}}=2]$
for a given $\underline{\bfs a}\in\FF_2$.
\begin{proposition}\label{prop: analysis c=2 fixed}
For $q>d$ and $\underline{\bfs a}:=(\bfs a_1,\bfs a_2)\in\FF_2$, we
have
$$\big|p_{r,d}[C_{\underline{\bfs a}}=2]
-\mu_d(1-\mu_d)\big|\le\frac{3}{q}.$$
\end{proposition}
\begin{proof}
Observe that
$$\{C_{\underline{\bfs
a}}=2\}=\{F\in\mathcal{F}_{r,d}: N_{1,d}(F(\bfs a_2,T))>0\}\setminus
\{F\in\mathcal{F}_{r,d}:N_{1,d}(F(\bfs a_1,T))>0\}.$$
The number of elements of $\mathcal{F}_{r,d}$ having $\fq$--rational
zeros in the vertical strip defined by $\bfs a_2$ is determined in
Theorem \ref{th: prob C=1}. Therefore, it remains to find the number
$N_{\underline{\bfs a},2}$ of elements of $\mathcal{F}_{r,d}$ having
$\fq$--rational zeros both in the vertical strips defined by $\bfs
a_1$ and $\bfs a_2$. We have
$$N_{\underline{\bfs a},2}=\bigg|\bigcup_{x\in
\fq}\bigcup_{y\in \fq}\{F\in\mathcal{F}_{r,d}:F(\bfs a_1,x)=F(\bfs
a_2,y)=0\}\bigg|.$$
Given sets $\mathcal{X}\subset\fq$ and $\mathcal{Y}\subset\fq$, we
denote
$$\mathcal{S}_{\underline{\bfs a}}(\mathcal{X},\mathcal{Y}):=\{F\in\mathcal{F}_{r,d}:
F(\bfs a_1,x)=F(\bfs a_2,y)=0\textrm{ for all
}x\in\mathcal{X}\textrm{\ and }y\in\mathcal{Y}\}.$$
Then the inclusion--exclusion principle implies
\begin{equation}
\label{eq: definition N_{a,2}} N_{\underline{\bfs a},2}=
\sum_{j=1}^q \sum_{k=1}^q(-1)^{j+k} \sum_{\mathcal{X}_j\subset \fq}
\sum_{\mathcal{Y}_k \subset \fq} \left| \mathcal{S}_{\underline{\bfs
a}}(\mathcal{X}_j,\mathcal{Y}_k)\right|.
\end{equation}
where the sum runs over all subsets $\mathcal{X}_j\subset\fq$ and
$\mathcal{Y}_k\subset\fq$ of $j$ and $k$ elements respectively.
\begin{claim}
$\dfrac{N_{\underline{\bfs
a},2}}{|\mathcal{F}_{r,d}|}=\big(P_1[C_1=1]\big)^2+\frac{q-1}{q^{2d+2}}\binom{q-1}{d}^2
=\big(P_1[C_1=1]\big)^2+\mathcal{O}(q^{-1})$.
\end{claim}
\begin{proof}[Proof of Claim] For $1\le j,k\le q$, let
$$\mathcal{N}_{j,k}:=
\sum_{\mathcal{X}_j\subset \fq} \sum_{\mathcal{Y}_k \subset \fq} |
\mathcal{S}_{\underline{\bfs a}}(\mathcal{X}_j,\mathcal{Y}_k)|.$$
We determine $\mathcal{N}_{j,k}$ according to whether one of the
following four cases occurs.

First suppose that $j,k\le d$. As $\bfs a_1\not=\bfs a_2$, the
equalities $F(\bfs a_1,x)=0, F(\bfs a_2,y)=0$ for all $x\in
\mathcal{X}_j$ and $y\in\mathcal{Y}_k$ impose $j+k$
linearly--independent conditions on the coefficients of
$F\in\mathcal{F}_{r,d}$. Therefore, $|\mathcal{S}_{\underline{\bfs
a}}(\mathcal{X}_j,\mathcal{Y}_k)|=q^{\dim \mathcal{F}_{r,d}-j-k}$,
which implies
\begin{align*}
\mathcal{N}_{j,k}=\sum_{\mathcal{X}_j\subset \fq}
\sum_{\mathcal{Y}_k \subset \fq} q^{\dim \mathcal{F}_{r,d}-j-k}
=\binom{q}{j}\binom{q}{k}q^{\dim \mathcal{F}_{r,d}-j-k}.
\end{align*}

The second case is determined by the conditions $j>d$ and $k\le d$.
If $j>d$ and $\mathcal{X}_j\subset\fq$ is a subset of cardinality
$j$, then the condition $F(\bfs a_1,x)=0$ is satisfied for every
$x\in \mathcal{X}_j$ if and only if $F(\bfs a_1,X_r)=0$. We may
express the latter by $d+1$ linearly--independent linear equations
on the coefficients of $F\in \mathcal{F}_{r,d}$. On the other hand,
the equalities $F(\bfs a_2,y)=0$ for all $y\in \mathcal{Y}_k$ impose
$k$ additional linearly--independent conditions on the coefficients
of $F$. We conclude that
\begin{align*}
\mathcal{N}_{j,k}&=\sum_{\mathcal{X}_j,\mathcal{Y}_k\subset\fq}
q^{\dim \mathcal{F}_{r,d}-(d+1)-k} =\binom{q}{j}\binom{q}{k}q^{\dim
\mathcal{F}_{r,d}-(d+1)-k}.
\end{align*}

The third case, namely $j\le d$ and $k>d$, is completely analogous
to the second one. Finally, when $j>d$ and $k>d$, the conditions
under consideration imply $F(\bfs a_1,X_r)=F(\bfs a_2,X_r)=0$. We
readily deduce that
$$\mathcal{N}_{j,k}=\binom{q}{j}\binom{q}{k}q^{\dim
\mathcal{F}_{r,d}-2d-1}.$$

From the expression for $\mathcal{N}_{j,k}$ of the four cases under
consideration we infer that
\begin{align*}
\frac{N_{\underline{\bfs
a},2}}{|\mathcal{F}_{r,d}|}=&\frac{1}{|\mathcal{F}_{r,d}|}
\sum_{j=1}^q \sum_{k=1}^q(-1)^{j+k} \mathcal{N}_{j,k}
\\=&\sum_{j=1}^d \sum_{k=1}^d(-1)^{j+k}\binom{q}{j}\binom{q}{k}q^{-j-k}
+2\sum_{j=1}^d
\sum_{k=d+1}^q(-1)^{j+k}\binom{q}{j}\binom{q}{k}q^{-j-(d+1)}\\
&+\sum_{j=d+1}^q
\sum_{k=d+1}^q(-1)^{j+k}\binom{q}{j}\binom{q}{k}q^{-2d-1}.
\end{align*}
By \eqref{eq: identities combinatorial numbers} and elementary
calculations we obtain
\begin{align*}
\frac{N_{\underline{\bfs
a},2}}{|\mathcal{F}_{r,d}|}=&\Bigg(\sum_{j=1}^d
(-1)^{j}\binom{q}{j}q^{-j}\Bigg)^2 -2\Bigg(\sum_{j=1}^d
(-1)^{j}\binom{q}{j}q^{-j}\Bigg)(-1)^d\binom{q-1}{d}q^{-d-1}
\\&+\binom{q-1}{d}^2q^{-2d-1}.
\end{align*}
This and Theorem \ref{th: prob C=1} readily imply the claim.
\end{proof}

Combining the previous claim and Theorem \ref{th: prob C=1} we
deduce that
\begin{align*}
p_{r,d}[C_{\underline{\bfs
a}}=2]&=P_1[C_1=1]-\frac{N_{\underline{\bfs
a},2}}{|\mathcal{F}_{r,d}|}\\&=\big(1-P_1[C_1=1]\big)P_1[C_1=1]-\frac{q-1}{q^{2d+2}}\binom{q-1}{d}^2.
\end{align*}
Let $g:\R\to\R$, $g(x):=(1-x)x$. The Mean Value theorem shows that
there exists $\xi\in(0,1)$ such that
$$\big(1-P_1[C_1=1]\big)P_1[C_1=1]-(1-\mu_d)\mu_d=g'(\xi)\,
\big(P_1[C_1=1]-\mu_d\big).$$
As the function $x\mapsto g'(x)$ maps the real interval $[0,1]$ to
$[-1,1]$, we conclude that $|g'(\xi)|\le 1$. Therefore, from
Corollary \ref{coro: prob C=1 - asymptotic} it follows that
$$\big|(1-P_1[C_1=1])P_1[C_1=1]-(1-\mu_d)\mu_d\big|\le
\big|P_1[C_1=1]-\mu_d\big|\le \frac{2}{q}.$$
On the other hand, it is easy to see that
$\frac{q-1}{q^{2d+2}}\binom{q-1}{d}^2\le 1/q$. This immediately
implies the statement of the proposition.
\end{proof}

Proposition \ref{prop: analysis c=2 fixed} is the critical step in
the analysis of the behavior of the probability $P_2[C_2=2]$, which
is estimated in the next result.
\begin{theorem}\label{th: prob C=2}
For any $q>d$,
$$\left|P_2[C_2=2]-(1-\mu_d)\mu_d
\right|\le \frac{3}{q}.$$
\end{theorem}
\begin{proof}
By Lemma \ref{lemma: p_2 in terms of C_r,d,a} and Proposition
\ref{prop: analysis c=2 fixed} we obtain
\begin{align*}
\left|P_2[C_2=2]-(1-\mu_d)\mu_d\right|&\le \frac{1}{N_2}
\sum_{\underline{\bfs a}\in\FF_2}\! \left|p_{r,d}[C_{\underline{\bfs
a}}=2] -(1-\mu_d)\mu_d\right|\le \frac{3}{q}.
\end{align*}
This finishes the proof of the theorem.
\end{proof}

We finish the section with a remark concerning the spaces considered
so far to discuss the probability that the SVS algorithm performs at
most two searches on vertical strips. For the analysis of the
probability of one search we have considered $\FF_1:=\fq^{r-1}$ and
the random variable $C_1:\FF_1\times
\mathcal{F}_{r,d}\to\{1,\infty\}$, while in the analysis of the
probability of two searches we have considered the random variable
$C_2:\FF_2\times \mathcal{F}_{r,d}\to\{1,2,\infty\}$. To link both
analyses, in Lemma \ref{lemma: consistency conditions} below we
prove that
$$
P_2[C_2=1]= P_1[C_1=1],
$$
which shows the consistency of the probability spaces underlying
Theorems \ref{th: prob C=1} and \ref{th: prob C=2}. In Section
\ref{section: analysis of SVS algorithm} we shall show that the
analysis of the probability that $s$ vertical strips are searched
can be done in a unified framework for any $s\ge 1$.
%
%
\section{The number of searches for given vertical strips}
\label{section: analysis fixed ordered VS}
As can be inferred from Section \ref{section: analysis C=1 and C=2},
a critical step in the probabilistic analysis of SVS algorithm is
the determination of the probability of $s$ searches, for a given
choice of $s$ vertical strips. The cases $s=1$ and $s=2$ were
discussed in Section \ref{section: analysis C=1 and C=2}. In this
section we carry out the analysis of the general case.

Fix $3\le s\le \min\{\binom{d+r-1}{r-1},q^{r-1}\}$ and
$\bfs{a}_1,\dots,\bfs{a}_s\in\fq^{r-1}$ with $\bfs a_i \not= \bfs
a_j$ for $i\not= j$. Denote $\underline{\bfs
a}:=(\bfs{a}_1,\dots,\bfs{a}_s)$. Assuming that $\underline{\bfs a}$
is the choice for the first $s$ vertical strips to be considered, we
analyze the probability that the SVS algorithm finds an
$\fq$--rational zero of the polynomial under consideration in the
$s$th search.

For this purpose, we consider the set $\mathcal{F}_{r,d}$ endowed
with the uniform probability $p_{r,d}$ and the random variable
$C_{\underline{\bfs a}}:=C_{\underline{\bfs
a},r,d}:\mathcal{F}_{r,d}\to\{1,2\klk s,\infty\}$ which counts the
number of searches for a given input on the vertical strips
determined by $\bfs{a}_1,\dots,\bfs{a}_s$, $C_{\underline{\bfs
a}}(F)=\infty$ meaning that $F$ has no $\fq$--rational zeros on
these vertical strips.

We start with the following elementary result.
\begin{lemma}\label{lemma: prob preimage linear map}
Let $\mathbb V$ and $\mathbb W$ be $\fq$--linear spaces of finite
dimension and $\Phi:\mathbb V\to \mathbb W$ any $\fq$--linear
mapping. Consider $\mathbb V$ and  $\mathbb W$ endowed with the
uniform probabilities $P_{\mathbb V}$ and $P_{\mathbb W}$
respectively. Then for any $A\subset {\mathbb W}$ we have
$$
P_{\mathbb V}(\Phi^{-1}(A))=\frac{|A\cap {\rm Im}(\Phi)|}{|{\rm
Im}(\Phi)|}=\frac{P_{\mathbb W}(A\cap {\rm Im}(\Phi))}{P_{\mathbb
W}( {\rm Im}(\Phi))}=:P_{\rm Im \Phi}(A).
$$
\end{lemma}

\begin{proof}
We have
$$\frac{1}{|\mathbb V|}|\Phi^{-1}(A)|=\frac{1}{|\mathbb
V|}\sum_{\bfs{w}\in A}|\Phi^{-1}(\bfs{w})|=\frac{1}{|\mathbb V|}
|{\rm Ker}(\Phi)|\,|A\cap{\rm Im(\Phi)}|.$$
By the Dimension theorem and the equality $|\mathbb S|= q^{\dim
\mathbb S}$, valid for any $\fq$--vector space $\mathbb{S}$, we see
that $|\mathbb{V}|=|{\rm Ker}(\Phi)|\,|{\rm Im(\Phi)}|$. Then
$$\frac{1}{|\mathbb V|}|\Phi^{-1}(A)|=\frac{|A\cap{\rm Im(\Phi)}|}{|{\rm
Im(\Phi)}|}=\frac{P_{\mathbb W}(A\cap {\rm Im}(\Phi))}{P_{\mathbb
W}( {\rm Im}(\Phi))}.$$
This finishes the proof of the lemma.
\end{proof}

For simplicity of notations, we replace the variable $X_r$ by a new
indeterminate $T$ and consider the $\fq$--linear mapping
$\Phi:=\Phi_{\underline{\bfs a}}:\mathcal{F}_{r,d}\to
\mathcal{F}_{1,d}^s $ defined as
\begin{equation}\label{eq: def Phi Lambda s}
\Phi(F):=\big(F(\bfs a_1, T),\ldots,F(\bfs a_s, T)\big).
\end{equation}
Since $\mathrm{Im}(\Phi)$ is an $\fq$--linear space, by Lemma
\ref{lemma: prob preimage linear map} it follows that
\begin{equation}\label{eq: prob C=s in terms of R_s}
p_{r,d}[C_{\underline{\bfs a}}=s]=
\frac{\big|(\{N=0\}^{s-1}\times\{N>0\}) \cap
\mathrm{Im}(\Phi)\big|}{|\mathrm{Im}(\Phi)|},
\end{equation}
where $N:=N_{1,d}$ denotes the random variable which counts the
number of zeros in $\fq$ of the elements of $\mathcal{F}_{1,d}$. As
a consequence, we need to estimate the quantity
$$R_s:=\big|\big(\{N=0\}^{s-1}\times\{N>0\}\big)
\cap \mathrm{Im}(\Phi)\big|.$$
In the next section we obtain a characterization of the image of
$\Phi$ that will allow us to express $R_s$ in terms of the average
cardinality of the value set of certain families of univariate
polynomials. This is the critical step to estimate the quantity
$R_s$.

As we explain below, there exists a unique positive integer
$\kappa_s\le d$ such that
$$\binom{\kappa_s+r-2}{r-1}< s\le \binom{\kappa_s+r-1}{r-1}.$$
In the sequel we shall assume that the points
$\bfs{a}_1,\dots,\bfs{a}_s$ under consideration satisfy the
condition we now state. For $1\le j\le \kappa_s$, let
$D_j:=\binom{j+r-1}{r-1}$ and denote by $\Omega_j:=\{\bfs
\omega_1\klk \bfs \omega_{D_j}\}\subset(\Z_{\ge 0})^{r-1}$ the set
of $(r-1)$--tuples $\bfs \omega_k:=(\omega_{k,1}\klk
\omega_{k,r-1})$ with $|\bfs \omega_k|:=\omega_{k,1}\plp
\omega_{k,r-1}\le j$. Let $\bfs a_i^{\bfs
\omega_k}:=a_{i,1}^{\omega_{k,1}}\cdots a_{i,r-1}^{\omega_{k,r-1}}$
for $1\le i\le s$ and $1\le k\le D_j$. Then we require that the
multivariate Vandermonde matrix
\begin{equation}\label{eq: multiv Vandermonde matrix M_j}
\mathcal{M}_j:= \left(
\begin{array}{ccc}
\bfs a_1^{\bfs\omega_1} & \cdots & \bfs a_1^{{\bfs\omega}_{D_j}}\\
\vdots & & \vdots \\ \bfs a_s^{\bfs\omega_1} & \cdots & \bfs
a_s^{{\bfs\omega}_{D_j}}
\end{array}
\right)\in\fq^{s\times D_j}\end{equation}
has maximal rank $\min\{D_j,s\}$ for $1\le j\le \kappa_s$.

We briefly argue that this is a mild requirement which is likely to
be satisfied by any ``reasonable'' choice of the elements $\bfs
a_1\klk\bfs a_s\in\fq^{r-1}$. Let $\bfs A_1\klk \bfs A_s$ be
$(r-1)$--tuples of indeterminates over $\cfq$, that is, $\bfs
A_i:=(A_{i,1}\klk A_{i,r-1})$ for $1\le i\le s$, and denote by
$\mathcal{V}_j$ the following ${\min\{D_j,s\}\times \min\{D_j,s\}}$
Vandermonde matrix with entries in $\fq[\bfs A_1\klk \bfs A_s]$:
$$\mathcal{V}_j:=
\left(
\begin{array}{ccc}
\bfs A_1^{\bfs\omega_1} & \cdots & \bfs A_1^{\bfs\omega_{\min\{D_j,s\}}}\\
\vdots & & \vdots \\ \bfs A_{\min\{D_j,s\}}^{\bfs\omega_1} & \cdots
& \bfs A_{\min\{D_j,s\}}^{\bfs\omega_{\min\{D_j,s\}}}
\end{array}
\right).$$
Assume that the numbering of $\Omega_j:=\{\bfs \omega_1\klk \bfs
\omega_{D_j}\}\subset(\Z_{\ge 0})^{r-1}$ is made according to
degrees, i.e., $|\bfs \omega_k|\le |\bfs \omega_l|$ whenever $k\le
l$. In particular, $\bfs\omega_1=(0\klk 0)$. By \cite[Theorem
1.5]{DaTa09} it follows that $\det\mathcal{V}_j$ is absolutely
irreducible, namely it is a nonzero irreducible element of
$\cfq[\bfs A_1\klk \bfs A_s]$, for $1\le j\le \kappa_s$. Let
$\delta_j$ denote the degree of $\det\mathcal{V}_j$. We have the
bound $\delta_j\le jD_j$. Then \cite[Theorem 5.2]{CaMa06} proves
that the number $\mathcal{N}_j$ of $(r-1)$--tuples $\bfs a_1\klk\bfs
a_s\in\fq^{r-1\mathcal{\mathcal{}}}$ annihilating
$\det\mathcal{V}_j$ satisfies the estimate
\begin{equation}\label{eq: estimate bad number of s strips}
|\mathcal{N}_j-q^{s(r-1)-1}|\le (\delta_j-1)(\delta_j-2)q^{s(r-1)
-\frac{3}{2}}+5\delta_j^{\frac{13}{3}}q^{s(r-1)-2}.
\end{equation}
Any choice of $\bfs a_1\klk\bfs a_s$ avoiding these
$\mathcal{N}_j=\mathcal{O}(q^{s(r-1)-1})$ tuples for $1\le j\le
\kappa_s$ will satisfy our requirements. Furthermore, many ``bad''
choices $\bfs a_1\klk\bfs a_s$ annihilating the polynomial
$\det\mathcal{V}_j$ for a given $j$ will also work, as other minors
of the Vandermonde matrix $\mathcal{M}_j$ of \eqref{eq: multiv
Vandermonde matrix M_j} may be nonsingular. In particular, for $s\le
r$ and $\bfs a_1\klk\bfs a_s$ affinely independent, our requirement
is satisfied.

Summarizing, denote $\mathcal{V}^s:=\prod_{j=1}^{\kappa_s}\det
\mathcal{V}_j\in\fq[\bfs A_1\klk \bfs A_s]$ and let
\begin{equation}\label{eq: definition B_s}
{\tt B}_s:=\{\underline{\bfs a}:=(\bfs a_1\klk\bfs
a_s)\in\fq{\!}^{s(r-1)}:\mathcal{V}^s(\underline{\bfs a})=0\}.
\end{equation}
Then $|{\tt B}_s|=\mathcal{O}(q^{s(r-1)-1})$ and all the results of
this section are valid for any $\underline{\bfs
a}\in\fq{\!}^{s(r-1)}\setminus{\tt B}_s$.
%
%
\subsection{A characterization of the image of $\Phi$}
In order to characterize the image ${\rm Im}(\Phi)$, we shall
express each element of $\mathcal{F}_{r,d}$ by its coordinates in
the standard monomial basis $\mathcal{B}$ of $\mathcal{F}_{r,d}$,
considering the monomial order we now define. Denote by
$\mathcal{B}_i$ the set of monomials of $\fq[X_1\klk X_{r-1}]$ of
degree at most $i$ for $0\le i\le d$, with the standard
lexicographical order defined by setting $X_1<X_2<\cdots <X_{r-1}$.
The basis $\mathcal{B}$ is considered with the order
$\mathcal{B}=\{X_r^d,X_r^{d-1}\mathcal{B}_1\klk
X_r\mathcal{B}_{d-1},\mathcal{B}_d\}$, where each set
$X_r^{d-i}\mathcal{B}_i$ is ordered following the order induced by
the one of $\mathcal{B}_i$. In other words, any $F\in
\mathcal{F}_{r,d}$ can be uniquely expressed as
$$F=\sum_{i=0}^dF_i(X_1\klk X_{r-1})X_r^i,$$
where each $F_i$ has degree at most $d-i$ for $0\le i\le d$. Then
the vector of coefficients $(F)_{\mathcal{B}}$ of $F$ in the basis
$\mathcal{B}$ is given by
$(F)_{\mathcal{B}}=\big((F_d)_{\mathcal{B}_0}\klk
(F_0)_{\mathcal{B}_d}\big)$. On the other hand, we shall express the
elements of $\mathcal{F}_{1,d}^s$ in the basis
$\mathcal{B}':=\{T^d\klk T,1\}^s$.

Let
$$
D_j:=\binom{j+r-1}{r-1}=|\mathcal{B}_j|\quad (0\le j\le d),\quad
D:=\binom{d+r}{r}=|\mathcal{B}|=\sum_{j=0}^d|\mathcal{B}_j|.$$
We also set $D_{-1}:=0$. Observe that the sequence $(D_j)_{j\ge -1}$
is strictly increasing. Therefore, for each $i$ with $1\le i\le s$
there exists a unique $\kappa_i\in\N$ such that
\begin{equation}\label{eq: definition kappa_i}
D_{\kappa_i-1}< i\le D_{\kappa_i}.
\end{equation}
The following remarks can be easily established.
\begin{remark}\label{rem: properties kappa_i}{}{}\
\begin{itemize}
\item $\kappa_i\le j$ if and only if $i\le D_j$.
  \item $\kappa_1=0$, $\kappa_s\le d$.
\end{itemize}
\end{remark}

The matrix $\M_\Phi\in\fq^{s(d+1)\times D}$ of $\Phi$ with respect
to the bases defined above can be written as the following block
matrix:
$$\M_\Phi=
\left(
  \begin{array}{c}
    \M_1 \\
    \vdots \\
    \M_s \\
  \end{array}
\right),
$$
where $\M_i\in\fq^{(d+1)\times D}$ is the diagonal block matrix
$$\M_i:=
\left(
  \begin{array}{cccc}
\M_{i,0} &  \\
         & \M_{i,1} \\
         &          &  \ddots  \\
         &          &         & \M_{i,d}
  \end{array}
\right),\quad \M_{i,j}:=\big(\bfs a_i^{\bfs\alpha}:|\bfs \alpha|\le
j\big)\in\fq^{1\times D_j}.$$
Our first result concerns the dimension of $\mathrm{Im}(\Phi)$.
\begin{lemma}\label{lemma: dim image Phi C=s}
For $s\leq \min\{D_d,q^{r-1}\}$, we have
\begin{align*}
\dim{\rm Im}(\Phi)=\binom{\kappa_s-1+r}{r}+s(d-\kappa_s+1) &=
\sum_{i=1}^s(d+1-\kappa_i).
\end{align*}
\end{lemma}
\begin{proof}
Let $\bfs{h}:=(h_1,\ldots,h_s)$ be an element of
$\mathrm{Im}(\Phi)$. Then there exists $F \in \mathcal{F}_{r,d}$
with $\bfs{h}=\Phi(F)$. Denote by
$(F)_{\mathcal{B}}=\big((F_d)_{\mathcal{B}_0}\klk
(F_0)_{\mathcal{B}_d}\big)$ the coordinates of $F$ in the basis
$\mathcal{B}$. Then the block structure of the matrix $\M_\Phi$
implies
\begin{equation}\label{eq: expression Phi(F)}
\Phi(F)=\sum_{j=0} ^d \left(
  \begin{array}{c}
    \M_{1,j} \\
    \vdots \\
    \M_{s,j} \\
  \end{array}
\right)(F_{d-j})_{\mathcal{B}_j}T^{d-j}.
\end{equation}

As $\underline{\bfs a}\notin{\tt B}_s$, we have
$$\mathrm{rank}\left(
  \begin{array}{c}
    \M_{1,j} \\
    \vdots \\
    \M_{s,j} \\
  \end{array}
\right)=\min\{D_j,s\}= \left\{\begin{array}{cl}
D_j&\textrm{ for }0\le j\le\kappa_s-1,\\
s&\textrm{ for }\kappa_s\le j\le d.
\end{array} \right.$$
As a consequence,
$$\dim\mathrm{Im}(\Phi)=\sum_{j=0}^{\kappa_{s-1}}D_j+
s(d-\kappa_s+1)=\binom{\kappa_s-1+r}{r}+s(d-\kappa_s+1).$$
This proves the first assertion of the lemma. To prove the second
assertion, we have
\begin{align*}
\sum_{i=1}^s(d+1-\kappa_i)&=\sum_{j=0}^{\kappa_s}\sum_{i=D_{j-1}+1}^{\min\{D_j,\,s\}}(d+1-j)
\\&=\sum_{j=0}^{\kappa_s-1}(d+1-j)(D_j-D_{j-1})+
(d+1-\kappa_s)(s-D_{\kappa_s-1}).
\end{align*}
Since
$\sum_{j=0}^k(D_j-D_{j-1})=D_k$,
we conclude that
$$
\sum_{i=1}^s(d+1-\kappa_i)=-\sum_{j=0}^{\kappa_s-1} j(D_j-D_{j-1})
+(d+1-\kappa_s)s+\kappa_s D_{\kappa_s-1}.
$$
Taking into account the identity
$\sum_{j=0}^Kj\binom{j+R}{R}=(R+1)\binom{R+1+K}{R+2}$,
we obtain
$$\sum_{i=1}^s(d+1-\kappa_i)=-(r-1)\binom{\kappa_s+r-2}{r}+
(d+1-\kappa_s)s+\kappa_s D_{\kappa_s-1}.$$
A simple calculation finishes the proof of the lemma.
\end{proof}


Next we determine a suitable parameterization of
$\mathrm{Im}(\Phi)$. To this end, let $\Phi^*:
\mathrm{Im}(\Phi)\to\fq^{\dim{\rm Im}(\Phi)}$ be the $\fq$--linear
mapping defined by
$$\Phi^*(\bfs{h}):=\bfs h^*,$$
where $\bfs{h}:=(h_1\klk h_s)$, $h_i:=(h_{d,i}\klk
h_{0,i})\in\fq^{d+1}$ for $1\le i\le s$ and
\begin{equation}\label{eq: definition h^*}
\bfs h^*:=(h_1^*\klk h_s^*),\quad h_i^*:=(h_{d-\kappa_i,i}\klk
h_{0,i})\ \ (1\le i\le s).
\end{equation}
Lemma \ref{lemma: dim image Phi C=s} shows that $\Phi^*$ is
well--defined.
\begin{lemma}\label{lemma: image Phi C=s}
$\Phi^*$ is an isomorphism.
\end{lemma}
\begin{proof}
Since $\Phi^*$ is a linear mapping between $\fq$--vector spaces of
the same dimension, it suffices to show that $\Phi^*$ is injective.
Fix $\bfs{h}:=\Phi(F)\in \mathrm{Im}(\Phi)$ with
$\bfs{h}^*=\bfs{0}$. From (\ref{eq: expression Phi(F)}) we deduce
that
\begin{equation}\label{eq: imagen Phi en coordenadas}
\left(
  \begin{array}{c}
    \M_{1,j} \\
    \vdots \\
    \M_{s,j} \\
  \end{array}
\right)(F_{d-j})_{\mathcal{B}_j}= \left(
  \begin{array}{c}
    h_{d-j,1} \\
    \vdots \\
    h_{d-j,s} \\
  \end{array}
\right).
\end{equation}

Fix $j$ with $0\le j\le \kappa_s-1$. Then the element $h_{d-j,i}$ is
included in the definition of $h_i^*$ if and only if $i\le D_j$ (see
Remark \ref{rem: properties kappa_i}). As $\bfs h^*=\bfs 0$ by
hypothesis, it follows that $h_{d-j,i}=0$ for $1\le i\le D_j$ and we
have the identity
$$\left(
  \begin{array}{c}
    \M_{1,j} \\
    \vdots \\
    \M_{D_j,j} \\
    \M_{D_j+1,j} \\
    \vdots \\
    \M_{s,j} \\
  \end{array}
\right)(F_{d-j})_{\mathcal{B}_j}= \left(
  \begin{array}{c}
    0 \\
    \vdots \\
    0 \\
    h_{d-j,D_j+1} \\
    \vdots \\
    h_{d-j,s} \\
  \end{array}
\right).$$
Since the upper $(D_j\times D_j)$--submatrix of the matrix in the
left--hand side is invertible, we conclude that
$(F_{d-j})_{\mathcal{B}_j}=\bfs 0$. This implies
$h_{d-j,D_j+1}=\cdots=h_{d-j,s}=0$. On the other hand, for $j\ge
\kappa_s$ the element $h_{d-j,i}$ is included in the definition of
$h_i^*$ for $1\le i\le s$ and therefore $h_{d-j,i}=0$ for $1\le i\le
s$. This shows that $\bfs h=\bfs 0$.
\end{proof}

Denote by $\Psi:=(\psi_1\klk\psi_s):\fq^{\dim{\rm Im}(\Phi)}\to
\mathrm{Im}(\Phi)$ the inverse mapping of $\Phi^*$. We need further
information concerning the mappings $\psi_i$.
\begin{lemma}\label{lemma: coeff defined by h_i*}
Let be given $h_i^*:=(h_{d-\kappa_i,i}\klk
h_{0,i})\in\fq^{d+1-\kappa_i}$ for $1\le i\le s$. Let $\bfs
h^*:=(h_1^*\klk h_s^*)\in\fq^{\dim{\rm Im}(\Phi)}$ and $\bfs
h:=\Psi(\bfs h^*)$. Denote
$$h_i:=\psi_i(\bfs h^*):=h_{d,i}\,T^d\plp
h_{d+1-\kappa_i,i}\,T^{d+1-\kappa_i}+h_{d-\kappa_i,i}\,T^{d-\kappa_i}\plp
h_{0,i}.$$
Then $h_{d,i}\klk h_{d+1-\kappa_i,i}$ are uniquely determined by
$h_1^*\klk h_{i-1}^*$.
\end{lemma}
\begin{proof}
Fix $k$ with $0\le k\le \kappa_i-1$. Write $\bfs h:=\Phi(F)$. In the
proof of Lemma \ref{lemma: dim image Phi C=s} we prove that
$$
\left(
  \begin{array}{c}
    \M_{1,k} \\
    \vdots \\
    \M_{D_k,k} \\
  \end{array}
\right)(F_{d-k})_{\mathcal{B}_k}= \left(
  \begin{array}{c}
    h_{d-k,1} \\
    \vdots \\
    h_{d-k,D_k} \\
  \end{array}
\right),
$$
where the $(D_k\times D_k)$--matrix in the left--hand side is
invertible. The element $h_{d-k,l}$ is included in the definition of
$h_l^*$ if and only if $l\le D_k$. Furthermore, we have $k\le
\kappa_i-1\le \kappa_{i-1}$. We conclude that the vector in the
right--hand side is uniquely determined by $h_1^*\klk h_{i-1}^*$,
and thus so is $(F_{d-k})_{\mathcal{B}_k}$. Therefore, the identity
$$\left(
  \begin{array}{c}
    \M_{1,k} \\
    \vdots \\
    \M_{i,k} \\
  \end{array}
\right)(F_{d-k})_{\mathcal{B}_k}= \left(
  \begin{array}{c}
    h_{d-k,1} \\
    \vdots \\
    h_{d-k,i} \\
  \end{array}
\right)$$
shows that the element $h_{d-k,i}$ is uniquely determined by
$h_1^*\klk h_{i-1}^*$. 
\end{proof}

We end this section with the following remark.
\begin{remark}\label{rem: equality leading coefficients}
For each $\boldsymbol{h}:=(h_1,\ldots,h_s) \in {\rm Im}(\Phi)$, we
have $h_{d,1}=\ldots=h_{d,s}$. Indeed, from (\ref{eq: expression
Phi(F)}) we deduce that
$$
\left(
  \begin{array}{c}
    \M_{1,0} \\
    \vdots \\
    \M_{s,0} \\
  \end{array}
\right)(F_{d})_{\mathcal{B}_0}= \left(
  \begin{array}{c}
    1 \\
    \vdots \\
    1 \\
  \end{array}
\right)(F_{d})_{\mathcal{B}_0}= \left(
  \begin{array}{c}
    h_{d,1} \\
    \vdots \\
    h_{d,s} \\
  \end{array}
\right).
$$
This implies $h_{d,1}=\ldots=h_{d,s}=(F_{d})_{\mathcal{B}_0}$. In
particular, the coefficient $h_{d,1}$ of the monomial $T^d$ in the
polynomial $h_1$ uniquely determines the coefficient $h_{d,j}$ of
the monomial $T^d$ in $h_j$ for $2\le j\le s$. \qed
\end{remark}
%
%
\subsection{The probability of $s$ searches in terms of
cardinalities of value sets}
\label{subsec: prob C=s in terms of value sets}
%
For $\underline{\bfs a}:= (\bfs{a}_1,\dots,\bfs{a}_s)
\in\fq{\!}^{s(r-1)}\setminus{\tt B}_s$ as before, we need to
estimate the quantity
$$R_s:=\big|\big(\{N=0\}^{s-1}\times\{N>0\}\big) \cap
\mathrm{Im}(\Phi)\big|.$$

According to Lemma \ref{lemma: image Phi C=s}, each element
$\bfs{h}\in\mathrm{Im}(\Phi)$ can be uniquely expressed in the form
$\bfs h=\Psi(\bfs h^*)$, where $\bfs h^*$ is defined as in
\eqref{eq: definition h^*}. 
Hence,
\begin{equation}\label{eq: first expression Rs}
R_s=\sum_{\bfs{h}^*\in\fq{\!}^{\dim{\rm Im}(\Phi)}}
\bfs{1}_{\{N=0\}^{s-1}\times\{N>0\}}\big(\Psi(\bfs{h}^*)\big),
\end{equation}
where $\bfs{1}_{\{N=0\}^{s-1}\times\{N>0\}}:\mathcal{F}_{1,d}^s\to
\{0,1\}$ denotes the characteristic function of the set
$\{N=0\}^{s-1}\times\{N>0\}$. By Lemma \ref{lemma: coeff defined by
h_i*}, the coordinate $\psi_i(\bfs h^*)$ depends only on $\bfs
h_i^*:=(h_1^*\klk h_i^*)$ for $1\le i\le s$. We shall therefore
write $\psi_i(\bfs h^*)$ as $\psi_i(\bfs h_i^*)$ for $1\le i\le s$,
with a slight abuse of notation.

First, we rewrite the expression (\ref{eq: first expression Rs}) for
$R_s$ in a suitable form for our purposes.
\begin{lemma}\label{lemma: expression for R_s}
Let $\bfs{h}:=(\sum_{j=0}^dh_{j,1}T^j\klk \sum_{j=0}^dh_{j,s}T^j)$
be an arbitrary element of $\mathrm{Im}(\Phi)$ and let
$\bfs{h}^*:=\Phi^*(\bfs h):=(h_1^*\klk h_s^*)\in\fq^{\dim{\rm
Im}(\Phi)}$ be defined as in \eqref{eq:
definition h^*}. 
For $s\le \min\{D_d,q^{r-1}\}$, the following identity holds:
$$R_s=\sum_{\stackrel{\scriptstyle h_1^*\in\fq^{d+1}}{N(\psi_1(\bfs
h_1^*))=0}} \cdots\sum_{\stackrel{\scriptstyle
h_{s-1}^*\in\fq^{d+1-\kappa_{s-1}}}{N(\psi_{s-1}(\bfs
h_{s-1}^*))=0}} \sum_{h_s^*\in\fq^{d+1-\kappa_s}}
\bfs{1}_{\{N>0\}}\big(\psi_s(\bfs{h}_s^*)\big).$$
%
\end{lemma}
\begin{proof}
We may rewrite \eqref{eq: first expression Rs} in the following way:
\begin{align*}
R_s&=\sum_{h_1^*\in\fq^{d+1}}\cdots
\sum_{h_s^*\in\fq^{d+1-\kappa_s}}
\bfs{1}_{\{N=0\}^{s-1}\times\{N>0\}}\big(\Psi(\bfs{h}^*)\big).
\end{align*}
%
As a consequence of the remarks before the statement of Lemma
\ref{lemma: expression for R_s}, it follows that
\begin{align*}
\bfs{1}_{\{N=0\}^{s-1}\times\{N>0\}}\big(\Psi(\bfs{h}^*)\big)&=
\prod_{i=1}^{s-1}\bfs 1_{\{N=0\}}\big(\psi_i(\bfs h^*)\big)\cdot
\bfs 1_{\{N>0 \}}\big(\psi_s(\bfs h^*)\big)\\
&=\prod_{i=1}^{s-1}\bfs 1_{\{N=0\}}\big(\psi_i(\bfs h_i^*)\big)\cdot
\bfs 1_{\{N>0 \}}\big(\psi_s(\bfs h_s^*)\big).
\end{align*}
Then the previous expression for $R_s$ can be rewritten as follows:
$$R_s=\!\!\!\!\sum_{h_1^*\in\fq^{d+1}} \!\!\!\!\bfs 1_{\{N=0\}}\big(\psi_1(\bfs
h_1^*)\big)\
\cdots\!\!\!\!\!\!\!\!\!\!\!\!\sum_{h_{s-1}^*\in\fq^{d+1-\kappa_{s-1}}}
\!\!\!\!\!\!\!\!\bfs 1_{\{N=0\}} \big(\psi_{s-1}(\bfs
h_{s-1}^*)\big)
\!\!\!\!\!\!\sum_{h_s^*\in\fq^{d+1-\kappa_s}}\!\!\!\!\!\!\!\!
\bfs{1}_{\{N>0\}}\big(\psi_s(\bfs{h}_s^*)\big),$$
%
%
which readily implies the lemma.
\end{proof}

For $1\le i\le s-1$, fix $h_i^*\in\fq^{d+1-\kappa_i}$. For each
$h_s^*:=(h_{d-\kappa_s,s}\klk h_{0,s})\in\fq^{d+1-\kappa_s}$, denote
by $f_{h_s^*}$ the polynomial
$$f_{h_s^*}:=\psi_s(h_1^*\klk h_s^*):=h_{d,s}T^d\plp
h_{d+1-\kappa_s,s}T^{d+1-\kappa_s}+h_{d-\kappa_s,s}T^{d-\kappa_s}\plp
h_{0,s}.$$
According to Lemma \ref{lemma: expression for R_s}, we are
interested in estimating the sum
\begin{equation}\label{eq: interpolation problem}
\sum_{h_s^*\in\fq^{d+1-\kappa_s}} \bfs{1}_{\{N>0\}}(f_{h_s^*}).
\end{equation}
For $h_s^*:=(h_{d-\kappa_s,s}\klk h_{0,s})\in\fq^{d+1-\kappa_s}$,
denote $\widehat{h}_s^*:=(h_{d-\kappa_s,s}\klk
h_{1,s})\in\fq^{d-\kappa_s}$ and
$f_{\widehat{h}_s^*}:=\sum_{j=1}^dh_{j,s}T^j=f_{h_s^*}-f_{h_s^*}(0)$.
We observe that
\begin{align}\sum_{h_s^*\in\fq^{d+1-\kappa_s}}
\bfs{1}_{\{N>0\}}(f_{h_s^*})=
\sum_{\widehat{h}_s^*\in\fq^{d-\kappa_s}} \sum_{h_{0,s}\in\fq}
\bfs{1}_{\{N>0\}}(f_{h_s^*})\nonumber
&=\sum_{\widehat{h}_s^*\in\fq^{d-\kappa_s}}
\mathcal{V}(f_{\widehat{h}_s^*})\\&=
\frac{1}{q}\sum_{h_s^*\in\fq^{d+1-\kappa_s}} \mathcal{V}(f_{h_s^*}),
\label{eq: last sum as average value set}
\end{align}
where  
$\mathcal{V}(f):=|\{f(c):c\in\fq\}|$ is the cardinality of the
value set of $f\in\fq[T]$. 
Lemma \ref{lemma: coeff defined by h_i*} proves that $h_{d,s}\klk
h_{d+1-\kappa_s,s}$ are uniquely determined by $\bfs
h^*_{s-1}:=(h_1^*\klk h_{s-1}^*)$. Thus, the sum in the right--hand
side of \eqref{eq: last sum as average value set} takes as argument
the cardinality of the value set of all the elements of
$\mathcal{F}_{1,d}$ having its first $\kappa_s$ coefficients
$(h_{d,s}\klk h_{d+1-\kappa_s,s})$ prescribed. Set
$\psi_s^{\mathrm{fix}}(\bfs h_{s-1}^*):=(h_{d,s}\klk
h_{d+1-\kappa_s,s})$ and denote 
%
\begin{equation}\label{eq: average value set of order s}
\mathcal{V}_d(\kappa_s,\psi_s^{\mathrm{fix}}(\bfs h_{s-1}^*)):=
\frac{1}{q^{d+1-\kappa_s}}\sum_{h_s^* \in\fq^{d+1-\kappa_s}}
\mathcal{V}(f_{h_s^*}).
\end{equation}

Now we express the probability that $C_{\underline{\bfs a}}=s$ in
terms of $\mathcal{V}_d(\kappa_s,\psi_s^{\mathrm{fix}}(\bfs
h_{s-1}^*))$.
\begin{lemma}\label{lemma: prob C=s with value sets}
For $s\le \min\{D_d,q^{r-1}\}$, the following identity holds:
$$p_{r,d}[C_{\underline{\bfs a}}=s]=
\frac{1}{q^{\scriptscriptstyle\sum\limits_{i=1}^{s-1}(d+1-\kappa_i)}}\sum_{\stackrel{\scriptstyle
h_1^*\in\fq^{d+1}}{N(\psi_1(\bfs h_1^*))=0}} \cdots\
\sum_{\stackrel{\scriptstyle
h_{s-1}^*\in\fq^{d+1-\kappa_{s-1}}}{N(\psi_{s-1}(\bfs
h_{s-1}^*))=0}}
\frac{\mathcal{V}_d(\kappa_s,\psi_s^{\mathrm{fix}}(\bfs
h_{s-1}^*))}{q}.$$
\end{lemma}
\begin{proof}
By Lemma \ref{lemma: dim image Phi C=s} we know that
$ \dim{\rm Im}(\Phi)= \sum_{i=1}^s(d+1-\kappa_i)$.
Combining this with \eqref{eq: prob C=s in terms of R_s} and Lemma
\ref{lemma: expression for R_s} we obtain
\begin{align*}
&p_{r,d}[C_{\underline{\bfs a}}=s]=\\
&\frac{1}{q^{\scriptscriptstyle\sum\limits_{i=1}^{s-1}(d+1-\kappa_i)}}\!\sum_{\stackrel{\scriptstyle
h_1^*\in\fq^{d+1}}{N(\psi_1(\bfs h_1^*))=0}} \cdots\
\sum_{\stackrel{\scriptstyle
h_{s-1}^*\in\fq^{d+1-\kappa_{s-1}}}{N(\psi_{s-1}(\bfs
h_{s-1}^*))=0}}
\frac{1}{q^{d+1-\kappa_s}}\sum_{h_s^*\in\fq^{d-\kappa_s+1}}
\bfs{1}_{\{N>0\}}\big(\psi_s(\bfs{h}_s^*)\big).
\end{align*}
Then \eqref{eq: last sum as average value set} and \eqref{eq:
average value set of order s} complete the proof of the lemma.
\end{proof}

If $s\le \min\{D_{d-2},q^{r-1}\}$, then, as we explain in the next
section, for any $\bfs h_{s-1}^*$ such that $f_{h_s^*}$ is of degree
$d$, the average cardinality in \eqref{eq: average value set of
order s} has the asymptotic behavior
$\mathcal{V}_d(\kappa_s,\psi_s^{\mathrm{fix}}(\bfs
h^*_{s-1}))=\mu_d\,q+\mathcal{O}(q^{1/2})$. Combining this with
Lemma \ref{lemma: prob C=s with value sets} we shall be led to
consider ``inner'' sums in the expression for
$p_{r,d}[C_{\underline{\bfs a}}=s]$, which shall be expressed in
terms of the average cardinality of the value sets of the families
of polynomials we now introduce. For $1\le i\le s-1$ and $1\le j\le
i-1$, fix $h_j^*:=(h_{d-\kappa_j,j}\klk
h_{0,j})\in\fq^{d+1-\kappa_j}$. For each
$h_i^*:=(h_{d-\kappa_i,i}\klk h_{0,i})\in\fq^{d+1-\kappa_i}$, denote
%
$$f_{h_i^*}:=\psi_i(h_1^*\klk h_i^*):=h_{d,i}T^d\plp
h_{d+1-\kappa_i,i}T^{d+1-\kappa_i}+h_{d-\kappa_i,i}T^{d-\kappa_i}\plp
h_{0,i}.$$
Lemma \ref{lemma: coeff defined by h_i*} proves that the
coefficients $h_{d,i}\klk h_{d-\kappa_i+1,i}$ are uniquely
determined by $\bfs h^*_{i-1}:=(h_1^*\klk h_{i-1}^*)$. Consequently,
we set $\psi_i^{\mathrm{fix}}(\bfs h_{i-1}^*):=(h_{d,i}\klk
h_{d+1-\kappa_i,i})$ and consider the average cardinality
%
\begin{equation}\label{eq: average value set of order i}
\mathcal{V}_d(\kappa_i,\psi_i^{\mathrm{fix}}(\bfs h_{i-1}^*)):=
\frac{1}{q^{d+1-\kappa_i}}\sum_{h_i^* \in\fq^{d+1-\kappa_i}}
\mathcal{V}(f_{h_i^*}).
\end{equation}
Our next result expresses the probability of $s$ searches in terms
of the quantities $\mathcal{V}_d(\kappa_i,\psi_i^{\mathrm{fix}}(\bfs
h_{i-1}^*))$ $(1\le i\le s)$.
\begin{theorem}\label{th: prob C=s with value sets}
For $s\le \min\{D_d,q^{r-1}\}$, we have
$$p_{r,d}[C_{\underline{\bfs a}}=s]=(1-\mu_d)^{s-1}\mu_d\,\frac{q-1}{q}
+\sum_{i=0}^s\mathcal{T}_i,$$
where $|\mathcal{T}_0|\le 1/q$,
\begin{align*}
\mathcal{T}_i:=&(1-\!\mu_d)^{s-i-1}\mu_d
\frac{q-1}{q^{\scriptscriptstyle\sum\limits_{j=1}^{i-1}(d+1-\kappa_j)}}
\!\!\!\!\!\sum_{\stackrel{\scriptstyle
h_1^*\in\fq^{d+1}}{\stackrel{\scriptstyle N(\psi_1(\bfs
h_1^*))=0}{\scriptstyle h_{d,1}=1}}} \!\!\!\cdots\!\!
\sum_{\stackrel{\scriptstyle
h_{i-1}^*\in\fq^{d+1-\kappa_{i-1}}}{N(\psi_{i-1}(\bfs
h_{i-1}^*))=0}}\!\!\!\!\!
\bigg(\mu_d-\frac{\mathcal{V}_d(\kappa_i,\psi_i^{\mathrm{fix}}(\bfs
h_{i-1}^*))}{q}\bigg)
\end{align*}
for $1\le i\le s-1$, and
$$
\mathcal{T}_s:=
\frac{q-1}{q^{\scriptscriptstyle\sum\limits_{i=1}^{s-1}(d+1-\kappa_i)}}
 \sum_{\stackrel{\scriptstyle
h_1^*\in\fq^{d+1}}{\stackrel{\scriptstyle N(\psi_1(\bfs
h_1^*))=0}{\scriptstyle h_{d,1}=1}}} \cdots
\sum_{\stackrel{\scriptstyle
h_{s-1}^*\in\fq^{d+1-\kappa_{s-1}}}{N(\psi_{s-1}(\bfs
h_{s-1}^*))=0}}\!\!
\bigg(\frac{\mathcal{V}_d(\kappa_s,\psi_s^{\mathrm{fix}}(\bfs
h_{s-1}^*))}{q}-\mu_d\bigg).
$$
\end{theorem}
\begin{proof}
Denote $C:=C_{\underline{\bfs a}}$. We split the expression for
$p_{r,d}[C=s]$ of Lemma \ref{lemma: prob C=s with value sets} into
two sums, depending on whether $h_{d,1}=0$ or not. More precisely,
we write
$$
p_{r,d}[C=s]= p_{r,d}[C=s,F_d=0]+p_{r,d}[C=s,F_d\not=0],$$
where
\begin{align*}
p_{r,d}[C=s,F_d=0]&=
\frac{1}{q^{\scriptscriptstyle\sum\limits_{i=1}^{s-1}(d+1-\kappa_i)}}
\!\sum_{\stackrel{\scriptstyle
h_1^*\in\fq^{d+1}}{\stackrel{\scriptstyle N(\psi_1(\bfs
h_1^*))=0}{\scriptstyle h_{d,1}=0}}}\! \cdots
\!\sum_{\stackrel{\scriptstyle
h_{s-1}^*\in\fq^{d+1-\kappa_{s-1}}}{N(\psi_{s-1}(\bfs
h_{s-1}^*))=0}}\!
\frac{\mathcal{V}_d(\kappa_s,\psi_s^{\mathrm{fix}}(\bfs h_{s-1}^*))}{q},\\
p_{r,d}[C=s,F_d\not=0]&=
\frac{1}{q^{\scriptscriptstyle\sum\limits_{i=1}^{s-1}(d+1-\kappa_i)}}\!\sum_{\stackrel{\scriptstyle
h_1^*\in\fq^{d+1}}{\stackrel{\scriptstyle N(\psi_1(\bfs
h_1^*))=0}{\scriptstyle h_{d,1}\not=0}}} \cdots
\!\sum_{\stackrel{\scriptstyle
h_{s-1}^*\in\fq^{d+1-\kappa_{s-1}}}{N(\psi_{s-1}(\bfs
h_{s-1}^*))=0}}
\!\frac{\mathcal{V}_d(\kappa_s,\psi_s^{\mathrm{fix}}(\bfs
h_{s-1}^*))}{q},\\
&=
\frac{q-1}{q^{\scriptscriptstyle\sum\limits_{i=1}^{s-1}(d+1-\kappa_i)}}\!\!\!\sum_{\stackrel{\scriptstyle
h_1^*\in\fq^{d+1}}{\stackrel{\scriptstyle N(\psi_1(\bfs
h_1^*))=0}{\scriptstyle h_{d,1}=1}}} \cdots
\!\sum_{\stackrel{\scriptstyle
h_{s-1}^*\in\fq^{d+1-\kappa_{s-1}}}{N(\psi_{s-1}(\bfs
h_{s-1}^*))=0}}
\frac{\mathcal{V}_d(\kappa_s,\psi_s^{\mathrm{fix}}(\bfs
h_{s-1}^*))}{q}.
\end{align*}

In the first term we consider the intersection of the $\fq$--linear
space $\mathrm{Im}(\Phi)$ with the linear subspace
$\mathcal{F}_{1,d-1}^s$. As the former is not contained in the
latter, the dimension of the intersection drops at least by one, and
Lemma \ref{lemma: prob preimage linear map} implies
$$\mathcal{T}_0:=p_{r,d}[C=s,F_d=0] \leq \frac{|{\rm Im}(\Phi)\cap
\mathcal{F}_{1,d-1}^s|}{|{\rm Im}(\Phi)|} \leq \frac{q^{\dim{\rm
Im}(\Phi)-1}}{q^{\dim{\rm Im}(\Phi)}}=\frac{1}{q}.$$
On the other hand, it is easy to see that the expression for
$p_{r,d}[C=s,F_d\not=0]$ may be rewritten in the following way:
$$p_{r,d}[C=s,F_d\not=0]=
\mu_d\,\frac{q-1}{q^{\scriptscriptstyle
\sum\limits_{i=1}^{s-1}(d+1-\kappa_i)}} \sum_{\stackrel{\scriptstyle
h_1^*\in\fq^{d+1}}{\stackrel{\scriptstyle N(\psi_1(\bfs
h_1^*))=0}{\scriptstyle h_{d,1}=1}}} \cdots\
\sum_{\stackrel{\scriptstyle
h_{s-1}^*\in\fq^{d+1-\kappa_{s-1}}}{N(\psi_{s-1}(\bfs
h_{s-1}^*))=0}} 1+\mathcal{T}_s,$$
where $\mathcal{T}_s$ is defined as in the statement of the theorem.
%

Now we claim that, for $1\le j\le s$,
$$p_{r,d}[C=s,F_d\not=0]=
(1-\mu_d)^{s-j}\mu_d\,
\frac{q-1}{q^{\scriptscriptstyle\sum\limits_{i=1}^{j-1}(d+1-\kappa_i)}}
\!\!\sum_{\stackrel{\scriptstyle
h_1^*\in\fq^{d+1}}{\stackrel{\scriptstyle N(\psi_1(\bfs
h_1^*))=0}{\scriptstyle h_{d,1}=1}}}\!\!\! \cdots
\!\!\!\sum_{\stackrel{\scriptstyle
h_{j-1}^*\in\fq^{d+1-\kappa_{j-1}}}{N(\psi_{j-1}(\bfs
h_{j-1}^*))=0}} \!\!\!1+\sum_{i=j}^s\mathcal{T}_i,$$
where $\mathcal{T}_i$ is defined as in the statement of the theorem.
The claim for $j=1$ is the assertion of the theorem.

We argue by downward induction on $j$ from $s$ to $1$, the case
$j=s$ being already proved. For $j<s$, suppose that the claim for
$j+1$ is already established. We have
\begin{align*}
\frac{1}{q^{d+1-\kappa_j}}\sum_{\stackrel{\scriptstyle
h_j^*\in\fq^{d+1-\kappa_j}}{N(\psi_j(\bfs h_j^*))=0}} 1
&=1-\frac{1}{q^{d+1-\kappa_j}}\sum_{\stackrel{\scriptstyle
h_j^*\in\fq^{d+1-\kappa_j}}{N(\psi_j(\bfs h_j^*))>0}} 1
=1-\frac{\mathcal{V}_d(\kappa_j,\psi_j^{\mathrm{fix}}(\bfs
h^*_{j-1}))}{q}.
\end{align*}
Replacing this identity in the expression for
$p_{r,d}[C=s,F_d\not=0]$ corresponding to the claim for $j+1$ we
readily deduce the claim for $j$, finishing thus the proof of the
theorem.
\end{proof}
%
%
\subsection{The probability of $C_{\underline{\bfs a}}=s$}
Theorem \ref{th: prob C=s with value sets} shows that the
probability that the SVS algorithm stops after $s\le D_d$ attempts
can be expressed in terms of the average cardinality
$\mathcal{V}_d(\kappa_i,\psi_i^{\mathrm{fix}}(\bfs h_{i-1}^*))$ of
the value set of certain families of univariate polynomials for
$1\le i\le s$.
%
%
Each of these families consists
%
of all the polynomials
$$f_{\bfs b}
:=\sum_{i=0}^{j-1}a_{d-i}T^{d-i}+\sum_{i=j}^db_{d-i}T^{d-i}$$
with $\bfs b:=(b_{d-j}\klk b_0)\in\fq^{d+1-j}$, for a given $1\le
j\le d$ and $\bfs a:=(a_d\klk a_{d-j})\in\fq^{j-1}$ with $a_d\not=0$
(due to Remark \ref{rem: equality leading coefficients}). We are
interested in the average
$$
\mathcal{V}_d(j,\bfs a):= \frac{1}{q^{d+1-j}}\sum_{\bfs
b\in\fq^{d+1-j}}\mathcal{V}(f_{\bfs b}).$$
%
%
Suppose that $q>d$. In \cite{CeMaPePr14}, the following estimate is
obtained for $1\le j\le {d}/{2}-1$:
\begin{equation}\label{eq: average value set - CMPP}
|\mathcal{V}_d(j,\bfs{a})-\mu_d\,q|\le \frac{e^{-1}}{2}+
\frac{(d-2)^5e^{2\sqrt{d}}}{2^{d-2}} +\frac{7}{q}.
\end{equation}
On the other hand, in \cite{MaPePr14} it is proved that, if the
characteristic $p$ of $\fq$ is greater than $2$ and $1 \leq j\leq
d-3$, then
\begin{equation}\label{eq: average value set - MPP}
\left|\mathcal{V}_d(j,\bfs a)-\mu_d\,q\right|\le d^2\,
2^{d-1}q^{\frac{1}{2}} +133\,d^{d+5} e^{2 \sqrt{d}-d}.
\end{equation}

Estimates \eqref{eq: average value set - CMPP} and \eqref{eq:
average value set - MPP} are the key point to determine the
asymptotic behavior of the right--hand side of the expression for
$p_{r,d}[C_{\underline{\bfs a}}=s]$ of Theorem \ref{th: prob C=s
with value sets}. More precisely, we have the following result.
\begin{theorem}\label{th: prob C=s for fix}
Let be given $\underline{\bfs a}:=(\bfs a_1\klk\bfs
a_s)\in\fq{\!}^{s(r-1)}\setminus{\tt B}_s$, where the set ${\tt
B}_s$ is defined in \eqref{eq: definition B_s}. For $s\le
 \min\big\{\binom{d/2+r-1}{r-1},q^{r-1}\big\}$, we have
$$
\left|p_{r,d}[C_{\underline{\bfs a}}=s]-(1-\mu_d)^{s-1}\mu_d\right|
\le \bigg({e^{-1}}+
\frac{(d-2)^5e^{2\sqrt{d}}}{2^{d-1}}+1\bigg)q^{-1} +14q^{-2} .
$$
On the other hand, if $p>2$ and $s\le
\min\big\{\binom{d+r-3}{r-1},q^{r-1}\big\}$, then
$$\left|p_{r,d}[C_{\underline{\bfs a}}=s]-(1-\mu_d)^{s-1}\mu_d\right|  \le
d^2 2^dq^{-\frac{1}{2}} +(266\,d^{d+5} e^{2 \sqrt{d}-d}+1)q^{-1}.$$
\end{theorem}
\begin{proof}
Suppose that $s\le \min\big\{\binom{d/2+r-1}{r-1},q^{r-1}\big\}$.
Then $\kappa_s\le d/2$, and thus $1\le \kappa_i-1\le {d}/{2}-1$ for
$1\le i\le s$. With notations as in Subsection \ref{subsec: prob C=s
in terms of value sets}, fix $1\le i\le s$ and
$h_j^*:=(h_{d-\kappa_j,j}\klk
h_{0,j})\in\fq{}^{\!\!d+1-\kappa_j}$ for $1\le j\le i-1$. 
Denote $\bfs h^*_{i-1}:=(h_1^*\klk h_{i-1}^*)$, 
set $\psi_i^{\mathrm{fix}}(\bfs h_{i-1}^*):=(h_{d,i}\klk
h_{d+1-\kappa_i,i})$ and consider the average cardinality
$ \mathcal{V}_d(\kappa_i,\psi_i^{\mathrm{fix}}(\bfs h_{i-1}^*))$ 
as in \eqref{eq: average value set of order s} or \eqref{eq: average
value set of order i}. By \eqref{eq: average value set - CMPP} we
conclude that, for any $\bfs h_{i-1}^*$ with $\deg f_{h_i^*}=d$,
$$
\left|\frac{\mathcal{V}_d(\kappa_i,\psi_i^{\mathrm{fix}}(\bfs
h^*_{i-1}))}{q}-\mu_d\right|\le \bigg(\frac{e^{-1}}{2}+
\frac{(d-2)^5e^{2\sqrt{d}}}{2^{d-2}}\bigg)q^{-1} +{7}q^{-2} .
$$
Further, defining $\mathcal{T}_i$ as in the statement of Theorem
\ref{th: prob C=s with value sets} for $1\le i\le s$, we obtain
\begin{align*}
|\mathcal{T}_i|&\le
(1-\mu_d)^{s-i-1}\mu_d\bigg(\bigg(\frac{e^{-1}}{2}+
\frac{(d-2)^5e^{2\sqrt{d}}}{2^{d-2}}\bigg)q^{-1}+{7}q^{-2}\bigg)\quad (1\le i\le s-1),\\
|\mathcal{T}_s|&\le\bigg(\frac{e^{-1}}{2}+
\frac{(d-2)^5e^{2\sqrt{d}}}{2^{d-2}}\bigg)q^{-1}+{7}q^{-2}.
\end{align*}
Therefore, the first assertion of the theorem follows from Theorem
\ref{th: prob C=s with value sets}.

On the other hand, for $s\le
\min\big\{\binom{d+r-3}{r-1},q^{r-1}\big\}$ we have $\kappa_s\le
d-2$, and hence $\kappa_i-1\le d-3$ for $1\le i\le s$. Therefore, if
$p>2$, then \eqref{eq: average value set - MPP} shows that
$$
\left|\frac{\mathcal{V}_d(\kappa_i,\psi_i^{\mathrm{fix}}(\bfs
h^*_{i-1}))}{q}-\mu_d\right|\le d^2\, 2^{d-1}q^{-\frac{1}{2}}
+133\,d^{d+5} e^{2 \sqrt{d}-d}q^{-1}.
$$
It follows that
\begin{align*}
|\mathcal{T}_i|&\le (1-\mu_d)^{s-i-1}\mu_d\big(d^2\,
2^{d-1}q^{-\frac{1}{2}} +133\,d^{d+5} e^{2
\sqrt{d}-d}q^{-1}\big)\quad (1\le i\le s-1),\\
|\mathcal{T}_s|&\le d^2\, 2^{d-1}q^{-\frac{1}{2}} +133\,d^{d+5} e^{2
\sqrt{d}-d}q^{-1}.
\end{align*}
This readily implies the second assertion of the theorem.
\end{proof}

We remark that the approach of the proof of Theorem \ref{th: prob
C=s for fix} cannot be applied to estimate the probability that $s>
s^*:=\binom{d+r-3}{r-1}$ vertical strips are searched, since the
behavior of the mapping $\Phi:=\Phi_{\underline{\bfs
a}}:\mathcal{F}_{r,d}\to\mathcal{F}_{1,d}^s$ of \eqref{eq: def Phi
Lambda s} may change significantly in this case. In what concerns
``large'' values of $s$, from Theorem \ref{th: prob C=s for fix} one
easily deduces the following result.
\begin{corollary}\label{coro: prob C greater than s}
With notations as in Theorem \ref{th: prob C=s for fix}, for
$s^*:=\min\big\{\binom{\frac{d}{2}+r-1}{r-1},q^{r-1}\big\}$ we have
$$p_{r,d}[C_{\underline{\bfs a}}> s^*]=(1-\mu_d)^{s^*}+
\mathcal{O}(q^{-1}).$$
On the other hand, if $p>2$ and $s^*:=
\min\big\{\binom{d+r-3}{r-1},q^{r-1}\big\}$, then
$$p_{r,d}[C_{\underline{\bfs a}}> s^*]=(1-\mu_d)^{s^*}+
\mathcal{O}(q^{-1/2}).$$
\end{corollary}

As $|1-\mu_d|\le 1/2$, from the expression of $s^*$ in both cases it
follows that the main term of this probability decreases
exponentially with $r$ and $d$.
%
%
\section{Probabilistic analysis of the SVS algorithm}
\label{section: analysis of SVS algorithm}
In this section we determine the average--case complexity of the SVS
algorithm. This analysis relies on the probability distribution of
the number of searches performed, which is the subject of the next
section.
%
%
\subsection{Probability distribution of the number of searches}
Similarly to Section \ref{section: analysis C=1 and C=2}, for $s\ge
3$ we denote
\begin{align*}
\FF_s:=\{(\bfs a_1\klk \bfs
a_s)\in\fq^{r-1}\times\cdots\times\fq^{r-1}:\bfs a_i\not=\bfs
a_j\textrm{ for }i\not=j\}, \quad N_s:=|\FF_s|,
\end{align*}
and consider the random variable $C_s:=C_{s,r,d}: \FF_s\times
\mathcal{F}_{r,d}\to\{1\klk s,\infty\}$ defined for $\underline{\bfs
a}:=(\bfs a_1\klk\bfs a_s)\in\FF_s$ and $F\in\mathcal{F}_{r,d}$ in
the following way:
$$C_s(\underline{\bfs
a},F):=\left\{\begin{array}{l} \min\{j: N_{1,d}(F(\bfs
a_j,X_r))>0\}\ \textrm{ if }\exists j\textrm{ with }N_{1,d}(F(\bfs
a_j,X_r))>0,\\[0.25ex]
\qquad\infty \qquad\textrm{otherwise}.\\
\end{array}\right.$$
We consider the set $\FF_s\times\mathcal{F}_{r,d}$ as before endowed
with the uniform probability $P_s:=P_{s,r,d}$ and analyze the
probability $P_s[C_s=s]$. To link the probability spaces determined
by $\FF_s\times \mathcal{F}_{r,d}$ and $P_s$ for $1\le s\le
q^{r-1}$, we have the following result.
\begin{lemma}\label{lemma: consistency conditions}
Let $s>1$ and let $\pi_s:\FF_s\times\mathcal{F}_{r,d}
\to\FF_{s-1}\times\mathcal{F}_{r,d}$ be the mapping induced by the
projection $\FF_s\to\FF_{s-1}$ on the first $s-1$ coordinates. If
$\mathcal{S}\subset\FF_{s-1}\times\mathcal{F}_{r,d}$, then
$P_s[\pi_s^{-1}(\mathcal{S})]=P_{s-1}[\mathcal{S}].$
\end{lemma}
\begin{proof}
Note that
\begin{align*}
\pi_s^{-1}(\mathcal{S})&= \bigcup_{F\in\mathcal{F}_{r,d}} \{(\bfs
a_1\klk \bfs a_s)\in\FF_s:(\bfs
a_1\klk \bfs a_{s-1},F)\in\mathcal{S}\}\times\{F\}\\
&=\bigcup_{F\in\mathcal{F}_{r,d}}
\bigcup_{\stackrel{\scriptstyle(\bfs a_1\klk\bfs
a_{s-1})\in\FF_{s-1}:}{(\bfs a_1\klk\bfs
a_{s-1},F)\in\mathcal{S}}}\{(\bfs a_1\klk\bfs a_{s-1})\}\times
(\fq^{r-1}\setminus\{\bfs a_1\klk\bfs a_{s-1}\})\times \{F\}.
\end{align*}
It follows that
\begin{align*}
P_s[\pi_s^{-1}(\mathcal{S})]&= \frac{1}{N_s|\mathcal{F}_{r,d}|}
\sum_{F\in\mathcal{F}_{r,d}} \sum_{\underline{\bfs
a}\in\FF_{s-1}:(\underline{\bfs
a},F)\in\mathcal{S}}(q^{r-1}-s+1)\\&=
\frac{1}{N_{s-1}|\mathcal{F}_{r,d}|}
\sum_{F\in\mathcal{F}_{r,d}}\big|\{\underline{\bfs
a}\in\FF_{s-1}:(\underline{\bfs a},F)\in\mathcal{S}\}\big|=
P_{s-1}[\mathcal{S}].
\end{align*}
This proves the lemma.
\end{proof}

According to the Kolmogorov extension theorem (see, e.g.,
\cite[Chapter IV, Section 5, Extension Theorem]{Feller71}), the
conditions of ``consistency'' of Lemma \ref{lemma: consistency
conditions} imply that the probabilities $P_s$ ($1\le s\le q^{r-1}$)
can be put in a unified framework. More precisely, we define
$\FF:=\FF_{\!q^{r-1}}$ and $P:=P_{q^{r-1}}$. Then the probability
measure $P$ defined on $\FF$ allows us to interpret consistently all
the results of this paper. In the same vein, the variables $C_s$
($1\le s\le q^{r-1}$) can be naturally extended to a random variable
$C: \FF\times\mathcal{F}_{r,d} \to\N\cup\{\infty\}$. Consequently,
we shall drop the subscript $s$ from the notations $P_s$ and $C_s$
in what follows.

For the analysis of the probability distribution of the number of
searches we express the probability $P[C=s]$ in terms of
probabilities concerning the random variables $C_{\underline{\bfs
a}}:=C_{\underline{\bfs a},r,d}: \mathcal{F}_{r,d}\to\N$,
$\underline{\bfs a}\in\FF_s$, which count the number of vertical
strips that are searched when the choice for the first $s$ vertical
strips is $\underline{\bfs a}$. As the result can be proved
following the proof of Lemma \ref{lemma: p_2 in terms of C_r,d,a}
{\em mutatis mutandis}, we state it without proof.
\begin{lemma}\label{lemma: p_s in terms of C_r,d,a}
We have
$$P[C=s]= \frac{1}{N_s}
\sum_{\underline{\bfs a}\in\FF_s} p_{r,d}[C_{\underline{\bfs a}}=s].
$$
\end{lemma}

In Theorem \ref{th: prob C=s for fix} we determine the asymptotic
behavior of $p_{r,d}[C_{\underline{\bfs a}}=s]$ for $\underline{\bfs
a}\in\FF_s\setminus {\tt B}_s$, where ${\tt B}_s\subset \FF_s$ is
the set of \eqref{eq: definition B_s}. By \eqref{eq: estimate bad
number of s strips} it follows that $|{\tt
B}_s|=\mathcal{O}(q^{s(r-1)-1})$, where the $\mathcal{O}$--constant
depends on $s$, $d$ and $r$, but is independent of $q$. Now, to
estimate the probability $P[C=s]$, Lemma \ref{lemma: p_s in terms of
C_r,d,a} implies
\begin{align*}
P[C=s]&= \frac{1}{N_s}\sum_{\underline{\bfs a} \in\FF_s\setminus{\tt
B}_s}p_{r,d}[C_{\underline{\bfs a}}=s]+\frac{1}{N_s}
\sum_{\underline{\bfs a}
\in{\tt B}_s}p_{r,d}[C_{\underline{\bfs a}}=s]\\
&=\frac{1}{N_s}\sum_{\underline{\bfs a} \in\FF_s\setminus{\tt
B}_s}p_{r,d}[C_{\underline{\bfs a}}=s]+\mathcal{O}(q^{-1}).
\end{align*}
As a consequence, from Theorem \ref{th: prob C=s for fix} we deduce
the following result.
\begin{theorem}\label{th: prob C=s for nind}
For $s\le \binom{d/2+r-1}{r-1}$, we have
$$P[C=s]=(1-\mu_d)^{s-1}\mu_d+\mathcal{O}(q^{-1}).
$$
On the other hand, if $p>2$ and $s\le \binom{d+r-3}{r-1}$, then
$$P[C=s]=
(1-\mu_d)^{s-1}\mu_d+\mathcal{O}(q^{-1/2}).$$
\end{theorem}
%
%
\subsection{Average--case complexity}
Now we are ready to determine the average--case complexity of the
SVS algorithm. 

Recall that, given $F\in\mathcal{F}_{r,d}$, the SVS algorithm
successively generates a sequence $\underline{\bfs a}:=(\bfs
a_1,\bfs a_2,\ldots,\bfs a_{q^{r-1}})\in \FF_{q^{r-1}}$, and
searches for $\fq$--rational zeros of $F$ in the vertical strips
$\{\bfs a_i\}\times\fq$ for $1\le i\le q^{r-1}$, until a zero of $F$
is found or all the vertical strips are exhausted. 
As discussed in Section \ref{section: intro}, the whole procedure
requires at most $C_{\underline{\bfs a}}(F)\cdot \tau(d,r,q)$
arithmetic operations in $\fq$, where
$\tau(d,r,q):=\mathcal{O}^\sim(D+d\log_2 q)$ is the maximum number
of arithmetic operations in $\fq$ necessary to perform a search in
an arbitrary vertical strip.

The SVS algorithm has a probabilistic routine which searches for
$\fq$--rational zeros of elements of $\mathcal{F}_{1,d}$, which
relies on $r_d$ random choices of elements of $\fq$, for certain
$r_d\in\N$. We denote by $\Omega_d:=\fq^{r_d}$ the set of all such
random choices and consider $\Omega_d$ endowed with the uniform
probability, $\FF\times\mathcal{F}_{r,d}$ with the (uniform)
probability$P$ of Section \ref{section: analysis of SVS algorithm},
and $\FF\times\mathcal{F}_{r,d}\times \Omega_d$ with the product
probability. Therefore, the cost of the SVS algorithm is represented
by the random variable $X:= X_{r,d} :
\FF\times\mathcal{F}_{r,d}\times \Omega_d:\rightarrow \N_{\geq 0}$
which counts the number $X(\underline{\bfs a},F,\omega)$ of
arithmetic operations performed on input $F\in\mathcal{F}_{r,d}$,
with the choice of vertical strips defined by $\underline{\bfs a}$
and the choice $\omega$ for the parameters of the routine for
univariate root finding.

We aim to determine the asymptotic behavior of the expected value of
$X$, namely
$$E[X]:=\frac{1}{|\FF||\mathcal{F}_{r,d}||\Omega_d|}
\sum_{(\underline{\bfs a},F,\omega)}X(\underline{\bfs a},F,\omega)
\le\frac{\tau(d,r,q)}{|\FF||\mathcal{F}_{r,d}|} \sum_{F\in
\mathcal{F}_{r,d}}\sum_{\underline{\bfs a}\in\FF}C({\underline{\bfs
a}},F).$$
%

We first study the case $r>2$, for which we have the following
result.
\begin{theorem}\label{th: average-case compl r>2}
Let $r>2$ and $s^*:=\binom{d/2+r-1}{r-1}$. Then the average--case
complexity of the SVS algorithm is bounded in the following way:
\begin{equation}\label{eq: estimate E(X) r ge 3}
E[X] \leq {\tau(d,r,q)} \big(\mu_d^{-1}+d(1-d^{-1})^{s^*}\big)
+\mathcal{O}(q^{-1/2}),
\end{equation}
where $\tau(d,r,q)$ is the cost of the search in a vertical strip.
\end{theorem}
\begin{proof}
Recall that an element of $\mathcal{F}_{r,d}$ is called relatively
$\fq$-irreducible if none of its irreducible factors over $\fq$ is
absolutely irreducible. Consider the sets
$$
A:=\{F\in \mathcal{F}_{r,d}:F\text{ is relatively
$\fq$-irreducible}\},\quad B:=\mathcal{F}_{r,d}\setminus A.
$$
We have
\begin{equation}\label{eq: sum in C(a,F)}
\sum_{F\in \mathcal{F}_{r,d}}\sum_{\underline{\bfs a}\in
\FF}C(\underline{\bfs a},F)=\sum_{F\in A}\sum_{\underline{\bfs a}\in
\FF}C(\underline{\bfs a},F)+\sum_{F\in B}\sum_{\underline{\bfs a}\in
\FF}C(\underline{\bfs a},F).
\end{equation}
By \cite[Corollary 6.7]{GaViZi13}, it follows that $
{|A|}/{|\mathcal{F}_{r,d}|}=\mathcal{O}\big(q^{\frac{-r(r-1)}{2}}\big)$.
Hence, we obtain
\begin{align}\label{eq: sum C(a,F) en A}
\frac{1}{|\FF||\mathcal{F}_{r,d}|}\sum_{F\in A}\sum_{\underline{\bfs
a}\in \FF}C(\underline{\bfs a},F)\leq
\frac{q^{r-1}}{|\mathcal{F}_{r,d}|}|A|=\mathcal{O}\big(q^{\frac{(r-1)(2-r)}{2}}\big)=\mathcal{O}(q^{-1}).
\end{align}

Next we study the second term in the right--hand side of \eqref{eq:
sum in C(a,F)}. We have
\begin{equation*}
\frac{1}{|\FF||\mathcal{F}_{r,d}|}\sum_{F\in B}\sum_{\underline{\bfs
a}\in \FF}C(\underline{\bfs a},F)= \frac{1}{|\mathcal{F}_{r,d}|}
\sum_{F\in B}\sum_{s=1}^{q^{r-1}}s\frac{|\{\underline{\bfs a}\in
\FF: C(\underline{\bfs a}, F)=s\}|}{|\FF|}.
\end{equation*}
From the conditions of consistency of Lemma \ref{lemma: consistency
conditions}, it follows that
\begin{align*}
\frac{1}{|\FF||\mathcal{F}_{r,d}|}\sum_{F\in B}\sum_{\underline{\bfs
a}\in \FF}C(\underline{\bfs
a},F)&=\frac{|B|}{|\mathcal{F}_{r,d}|}\sum_{s=1}^{q^{r-1}}s\frac{1}{|B|}\sum_{F\in
B}\frac{|\{\underline{\bfs a}\in \FF_s: C(\underline{\bfs a},
F)=s\}|}{|\FF_s|}\\  &=
\frac{|B|}{|\mathcal{F}_{r,d}|}\sum_{s=1}^{q^{r-1}}sP_{\FF\times
B}[C=s],
\end{align*}
where $P_{\FF\times B}$ denotes the uniform probability in
$\FF\times B$.

For $s\le s^*$, Theorem \ref{th: prob C=s for nind} allows us to
estimate the probability of $[C=s]$. Therefore, we decompose the sum
above in the following way:
\begin{align}\label{eq: sum C(a,F) en B}
\nonumber\sum_{s=1}^{q^{r-1}}\!sP_{\FF\times
B}[C=\!s]&
=\sum_{s=1}^{s^{*}}sP_{\FF\times B}[C=s]+
(s^{*}+1)\sum_{s=s^{*}+1}^{q^{r-1}}P_{\FF\times
B}[C=s]\\\nonumber&\qquad\qquad\qquad\quad\qquad+\sum_{s=s^{*}+2}^{q^{r-1}}(s-s^{*}-1)P_{\FF\times
B}[C=s]
\\&=\sum_{s=1}^{s^*}\!sP_{\FF\times B}[C=s]+
(s^*\!\!+\!1)P_{\FF\times B}[C\ge
s^*\!\!+\!1]+\!\!\!\!\sum_{s=s^*\!+2}^{q^{r-1}}\!\!\!\!P_{\FF\times
B}[C\ge \!s].
\end{align}

First we estimate the sum $S_1$ of the first two terms in the
right--hand of \eqref{eq: sum C(a,F) en B}. Arguing as in Lemma
\ref{lemma: p_s in terms of C_r,d,a}, we see that
\begin{align*}
P_{\FF\times B}[C=s]&
=\frac{1}{|\sf{F}_s|}\sum_{\underline{\bfs{a}}\in
\sf{F}_s}p_B[C_{\underline{\bfs{a}}}=s].
\end{align*}
From Theorem \ref{th: prob C=s for nind} and Corollary \ref{coro:
prob C greater than s} we have
\begin{align*}
S_1& =
\sum_{s=1}^{s^{*}}s(\mu_d(1-\mu_d)^{s-1}+\mathcal{O}(q^{-1}))+
(s^{*}+1)(1-\mu_d)^{s^*}+\mathcal{O}(q^{-1})\\
& = \mu_d\sum_{s=1}^{s^{*}}s(1-\mu_d)^{s-1}+
(s^{*}+1)(1-\mu_d)^{s^*}+\mathcal{O}(q^{-1}).
\end{align*}
Taking into account that $\sum_{n\geq 1}n z^{n-1}=1/(1-z)^2$ for any
$|z|\leq 1$, we obtain
\begin{align}
S_1&= \frac{1}{\mu_d}- \mu_d\sum_{s\ge s^{*}+1}s(1-\mu_d)^{s-1}+
(s^{*}+1)(1-\mu_d)^{s^*}+\mathcal{O}(q^{-1})
=\frac{1}{\mu_d}+\mathcal{O}(q^{-1}), \label{eq: estimate S1}
\end{align}
where the last inequality follows from the identity $\sum_{s\geq
s^*+1} sz^{s-1}=z^{s^*}(s^*+1-zs^*)/(1-z)^2$, which holds for any
$|z|<1$ (see, e.g., \cite[\S 2.3]{GrKnPa94}).

Next, we estimate the second sum $S_2$ of the right--hand of
\eqref{eq: sum C(a,F) en B}.
Observe that
$$p_B[C_{\underline{\bfs{a}}}\geq s]=
p_B[F\in B: N_{1,d}(F(\bfs{a}_i,X_r))=0\ (1\leq i \leq s-1)].$$
Hence,
\begin{align*}
S_2&\leq \frac{1}{|B|}\sum_{s=s^{*}+2}^{q^{r-1}}\frac{1}{|\sf{F}_s|}
\!\sum_{(\underline{\bfs a},\bfs{a}_s)\in \sf{F}_{s-1}\times \fq^{r-1}}\!\!
|\{F\in B: N_{1,d}(F(\bfs{a}_i,X_r))=0\,\,(1\leq i \leq s-1)\}|\\
&\leq \frac{q^{r-1}}{|B|}\sum_{s=s^{*}+2}^{q^{r-1}}
\frac{1}{q^{r-1}-(s-1)} \sum_{\underline{\bfs a}\in \sf{F}_{s-1}}
\mathop{\sum_{F\in B }}_{N_{1,d}(F(\bfs{a}_i,X_r))=0\,\,(1\leq i
\leq s-1)} \frac{1}{|\sf{F}_{s-1}|}\\
&\leq
\frac{q^{r-1}}{|B|}\sum_{s=s^{*}+2}^{q^{r-1}}\frac{1}{q^{r-1}-(s-1)}
\sum_{F\in B}P_{{\sf{F}}_{s-1}}[N_{1,d}=0],
\end{align*}
where
$P_{{\sf{F}}_{s-1}}[N_{1,d}=0]:=P_{{\sf{F}}_{s-1}}[\{\underline{\bfs
a}\in {\sf{F}}_{s-1}\,:N_{1,d}(F(\bfs{a}_i,X_r))=0,\,\,1\leq i \leq
s-1\}]$. As $N_{1,d}=0$ follows an hypergeometric distribution, the
probability $P_{{\sf{F}}_{s-1}}[N_{1,d}=0]$ can be expressed in the
following way (see, e.g., \cite[Chapter 6]{Feller68}):
$$
P_{{\sf{F}}_{s-1}}[N_{1,d}=0]=\frac{\binom{q^{r-1}-NS(F)}{s-1}}{\binom{q^{r-1}}{s-1}}.
$$
We deduce that
\begin{equation}\label{eq: estimate S2}
S_2\leq  \frac{1}{|B|}\sum_{s=s^{*}+2}^{q^{r-1}} \sum_{F\in
B}\bigg(1-\frac{NS(F)-1}{q^{r-1}-1}\bigg)^{s-1}.
\end{equation}

Fix $F\in B$. Then $F$ has  at least an absolutely irreducible
factor defined over $\fq$. Hence, for $q>d^4$, by \cite[Theorem
5.2]{CaMa06} it follows that $NS(F)\geq
\frac{q^{r-1}}{d}(1-\alpha)$, with $\alpha:=d^2q^{-1/2}$. This
implies
$$
1-\frac{NS(F)-1}{q^{r-1}-1}= 1-\frac{1-\alpha}{d}
+\mathcal{O}\big(q^{1-r}\big).
$$
Combining this inequality with \eqref{eq: estimate S2} we conclude
that
\begin{align}
S_2&\leq  \frac{1}{|B|}\sum_{s=s^*+2}^{q^{r-1}} \sum_{F\in
B}\big(1-(1-\alpha)d^{-1}+\mathcal{O}(q^{1-r})\big)^{s-1}\nonumber\\&=
\sum_{s=s^*+2}^{q^{r-1}}
\big(1-(1-\alpha)d^{-1}+\mathcal{O}(q^{1-r})\big)^{s-1}\nonumber\\&=
\frac{\big(1-(1-\alpha)d^{-1}\big)^{s^*+1}}{(1-\alpha)d^{-1}}+\mathcal{O}(q^{1-r})=
d(1-d^{-1})^{s^*+1}+\mathcal{O}(q^{-1/2}).\nonumber
\end{align}
Combining \eqref{eq: sum in C(a,F)}, \eqref{eq: sum C(a,F) en A} and
\eqref{eq: estimate S1} with this inequality, we deduce \eqref{eq:
estimate E(X) r ge 3}.
\end{proof}

Since $s^*>{d^2}/{4}$, the term $d(1-d^{-1})^{s^*+1}$ tends to zero
as $d$ and $r$ grow, and therefore the right--hand side of
\eqref{eq: estimate E(X) r ge 3} behaves as
${\mu_d}^{-1}\tau(d,r,q)$. We may paraphrase this as saying that, on
average, at most ${\mu_d}^{-1}\approx 1.58\ldots$ vertical strips
are searched until an $\fq$--rational point of the input polynomial
is obtained. For perspective, we remark that the probabilistic
algorithms of \cite{GaShSi03} (for bivariate polynomials) and
\cite{CaMa06a} and \cite{Matera10} (for $r$--variate polynomials)
propose $d$ searches in order to achieve a probability of success
greater than 1/2.

Now we analyze the average--case complexity $E[X]$ for $r=2$, that
is,
$$E[X]:=\frac{1}{|\FF||\mathcal{F}_{2,d}||\Omega_d|}
\sum_{(\underline{\bfs a},F,\omega)}X(\underline{\bfs a},F,\omega)
\le\frac{\tau(d,r,q)}{|\FF||\mathcal{F}_{2,d}|} \sum_{F\in
\mathcal{F}_{r,d}}\sum_{\underline{\bfs a}\in\\F}C({\underline{\bfs
a}},F).$$

For a real $0<\alpha<1$ to be determined, we consider the subsets
\begin{align*}
A&:=\{F\in \mathcal{F}_{2,d}: NS(F)\leq (1-\alpha)NS(2,d)\},\\
B&:=\{F\in \mathcal{F}_{2,d}: NS(F)> (1-\alpha)NS(2,d)\},
\end{align*}
where $NS(F)$ is the number of vertical strips on which $F$ has
$\fq$--rational zeros, and $NS(2,d)$ is the average number of such
vertical strips. We have
\begin{equation}\label{eq: sum over F2d r=2}
\sum_{F\in \mathcal{F}_{2,d}}\sum_{\underline{\bfs a}\in \FF}
C(\underline{\bfs a},F)=
\sum_{F\in A}\sum_{\underline{\bfs a}\in \FF}C(\underline{\bfs a},F)+\sum_{F\in B}
\sum_{\underline{\bfs a}\in \FF}C(\underline{\bfs a},F).\end{equation}

To estimate the first term of the right--hand of \eqref{eq: sum over
F2d r=2}, we start with an estimate for $|A|$. For this purpose,
according to Lemma \ref{lemma: average number vertical strips} and
Proposition \ref{prop: variance number of vert strips} below, the
mean $NS(2,d)$ and the variance $NS_2(2,d)$ of $NS(\cdot)$ have the
asymptotic behavior $NS(2,d)=\mu_d\,q+\mathcal{O}(1)$ and
$NS_2(2,d)=((d!)^{-2}+\mu_d(1-\mu_d))q+\mathcal{O}(1)$ respectively.
Then the Chebyshev inequality (see Corollary \ref{coro: NS(F)
deviating from NS(d,r)} below) implies
%
%
$$|A|\leq \bigg(\frac{1}{(\alpha\,\mu_d\,d!)^2}+\frac{1-\mu_d}{\alpha^2\mu_d}\bigg)q^{\dim \mathcal{F}_{2,d}-1}
+\mathcal{O}(q^{\dim \mathcal{F}_{2,d}-2}).$$
It follows that
\begin{equation} \label{eq: sum 1 r=2}
\frac{1}{|\FF||\mathcal{F}_{2,d}|} \sum_{F\in
A}\sum_{\underline{\bfs a}\in \FF} C(\underline{\bfs a},F)\leq
\frac{|A|q}{|\mathcal{F}_{2,d}|}\leq
 \bigg(\frac{1}{(\alpha\,\mu_d\,d!)^2}+\frac{1-\mu_d}{\alpha^2\mu_d}\bigg)
+\mathcal{O}(q^{-1}).
\end{equation}

Next we study the second sum in the right--hand side of \eqref{eq:
sum over F2d r=2}. Arguing as in the case $r>2$, for $s^*:=d/2+1$ we
obtain
\begin{align}
\frac{1}{|\FF||\mathcal{F}_{2,d}|}
\sum_{F\in B}\sum_{\underline{\bfs a}\in \FF}C(\underline{\bfs a},F)
&\leq \frac{1}{\mu_d} + \frac{1}{|B|}\sum_{s=s^{*}+2}^{q}\sum_{F\in B}
\bigg(1-\frac{NS(F)-1}{q-1}\bigg)^{s-1}+\mathcal{O}(q^{-1}).\nonumber
\end{align}
Fix $F\in B$. By definition $NS(F)>(1-\alpha)NS(2,d)$ and, according
to Lemma \ref{lemma: average number vertical strips} below, we have
$NS(2,d)=\mu_d\,q+\mathcal{O}(1)$. Hence, we obtain
$$
1-\frac{NS(F)-1}{q-1} \leq 1-(1-\alpha)\mu_d+\mathcal{O}(q^{-1}).
$$
Therefore,
$$
\frac{1}{|B|}\sum_{s=s^{*}+2}^{q} \sum_{F\in
B}\bigg(1-\frac{NS(F)-1}{q-1}\bigg)^{s-1}\leq
\frac{(1-(1-\alpha)\mu_d)^{s^*+1}}{(1-\alpha)\mu_d}+
\mathcal{O}(q^{-1}).
$$

Combining \eqref{eq: sum over F2d r=2} and \eqref{eq: sum 1 r=2}
with this inequality, we conclude that
$$
E[X]\leq \tau(d,r,q) \bigg(\frac{1}{\alpha^2}
\bigg(\frac{1-\mu_d}{\mu_d}+\frac{1}{(d!)^2\mu_d^2}\bigg)+\frac{1}{\mu_d}
+\big(1-(1-\alpha)\mu_d\big)^{s^*+1}\bigg)+\mathcal{O}(q^{-1}).
$$
Fixing $\alpha^*:=1-1/\sqrt{s^*}$, we obtain the following result.
\begin{theorem}
Let $r:=2$, $s^*:=d/2+1$ and $\alpha^*:=1-1/\sqrt{s^*}$. The
average--case complexity of the SVS algorithm is bounded in the
following way:
$$E[X]\leq \tau(d,r,q)\bigg(\frac{1}{\alpha^*{}^2}
\bigg(\frac{1-\mu_d}{\mu_d}+\frac{1}{(d!)^2\mu_d^2}\bigg)+\frac{1}{\mu_d}
+\Big(1-\mbox{$\frac{\mu_d}{\sqrt{s^*}}$}\Big)^{s^*+1}\bigg)+\mathcal{O}(q^{-1}),$$
where $\tau(d,r,q)$ is the cost of the search in a vertical strip.
\end{theorem}

As $d$ grows, the quantity $s^*$ tends to infinity and the
expression parenthesized in $E[X]$ tends to $(2-\mu_d)/\mu_d\approx
2.16\ldots$ This is an upper bound for the number of vertical strips
that are searched on average for $r=2$.
%
%
\section{On the probability distribution of the outputs}
\label{section: entropy}
This section is devoted to the analysis of the probability
distribution of the outputs of the SVS algorithm. For this purpose,
following \cite{BePa11} (see also \cite{BeLe12}), we use the concept
of Shannon entropy. For $F\in\mathcal{F}_{r,d}$, denote
$Z(F):=\{\bfs x\in\fq^r:F(\bfs x)=\bfs 0\}$ and $N(F):=|Z(F)|$. We
define a Shannon entropy $H_F$ associated with $F$ as
\begin{equation}\label{eq: definition entropy F}
H_F:=\sum_{\bfs x\in Z(F)}-P_{\bfs x,F} \log(P_{\bfs x,F}),
\end{equation}
where $P_{\bfs x,F}$ is the probability that the SVS algorithm
outputs $\bfs x$ on input $F$ and $\log$ denotes the natural
logarithm. It is well--known that $H_F\le\log N(F)$, and equality
holds if and only if $P_{\bfs x,F}=1/N(F)$ for every $\bfs x\in
Z(F)$. We shall consider the average entropy when $F$ runs through
all the elements of $\mathcal{F}_{r,d}$, namely
\begin{equation}\label{eq: definition entropy}
H:=\frac{1}{|\mathcal{F}_{r,d}|}\sum_{F\in \mathcal{F}_{r,d}}H_F.
\end{equation}

For an ``ideal'' algorithm for the search of $\fq$--rational zeros
of elements of $\mathcal{F}_{r,d}$, from the point of view of the
probability distribution of outputs, and $F\in\mathcal{F}_{r,d}$,
the probability $P_{\bfs x,F}^{\rm ideal}$ that a given $\bfs x\in
Z(F)$ occurs as output is equal to $1/N(F)$. As a consequence,
according to the definition \eqref{eq: definition entropy F}, the
corresponding entropy is
$$H^{\rm ideal}_F:=\sum_{\bfs x\in Z(F)}-P_{\bfs x,F}^{\rm ideal}
\log(P_{\bfs x,F}^{\rm ideal})= \sum_{\bfs x\in Z(F)}\frac{\log
N(F)}{N(F)}=\log N(F).$$
By the concavity of the function $x\mapsto \log x$ and \eqref{eq:
average number zeros}, we conclude that
\begin{equation}\label{eq: bound entropy ideal algorithm}
H^{\rm ideal}:=\frac{1}{|\mathcal{F}_{r,d}|}\sum_{F\in
\mathcal{F}_{r,d}} H^{\rm ideal}_F\le \log\left(\frac{\sum_{F\in
\mathcal{F}_{r,d}}N(F)}{|\mathcal{F}_{r,d}|}\right)=\log(q^{r-1}),
\end{equation}
where the last identity is due to \eqref{eq: average number zeros}.
In our analysis below, we shall exhibit a lower bound on the average
entropy $H$ which nearly matches this upper bound.
%
%
\subsection{On the number of vertical strips}
\label{subsec: average number vertical strips}
A critical point in the study of the behavior of $H$ is the analysis
of the probability distribution of the random variable
$NS:\mathcal{F}_{r,d}\to\Z_{\ge 0}$ which counts the number of
vertical strips with $\fq$--rational zeros of the elements of
$\mathcal{F}_{r,d}$.

Recall that $VS(F)$ denotes the set of vertical strips where each
$F\in\mathcal{F}_{r,d}$ has $\fq$--rational zeros and $NS(F)$ is its
cardinality, that is,
$$VS(F):=\{\bfs a\in\fq^{r-1}:(\exists\, x_r\in\fq)\ F(\bfs
a,x_r)=0\},\quad NS(F):=|VS(F)|.$$
We start considering the average number of vertical strips in
$\mathcal{F}_{r,d}$, namely
$$
NS({r,d}):=\frac{1}{|\mathcal{F}_{r,d}|}\sum_{F\in
\mathcal{F}_{r,d}}NS(F).
$$
According to \eqref{eq: prob C=1 in terms of number vert strips}, we
have
$NS({r,d})=q^{r-1}P[C=1]$. Therefore, as an immediate consequence of
Theorem \ref{th: prob C=1} and Corollary \ref{coro: prob C=1 -
asymptotic} we have the following result.

\begin{lemma}\label{lemma: average number vertical strips}
The number $NS({r,d})$ satisfies
\begin{align*}
NS({r,d})&=\sum_{k=1}^d(-1)^{k-1}\binom{q}{k}q^{r-1-k}
+(-1)^d\binom{q-1}{d}q^{r-d-2}\\
&=\mu_d\,q^{r-1}+\mathcal{O}(q^{r-2}).
\end{align*}
\end{lemma}

Next we determine the variance $NS_2({r,d})$ of the random variable
$NS(\cdot)$, that is,
$$NS_2({r,d}):=\frac{1}{|\mathcal{F}_{r,d}|}\sum_{F\in
\mathcal{F}_{r,d}}\big(NS(F)-NS({r,d})\big)^2=\frac{1}{|\mathcal{F}_{r,d}|}
\sum_{F\in \mathcal{F}_{r,d}}NS(F)^2- NS({r,d})^2.$$
\begin{proposition}\label{prop: variance number of vert strips}
The variance $NS_2({r,d})$ satisfies
$$
NS_2({r,d})=
\frac{1}{(d!)^2}q^{2r-3}+\mu_d(1-\mu_d)\,q^{r-1}+\mathcal{O}(q^{2r-4}).$$
\end{proposition}
\begin{proof}
Recall the notations $\FF_2:=(\fq^{r-1})^2\setminus \{(\bfs a,\bfs
a):\bfs a\in\fq^{r-1}\}$ and $N_2:=|\FF_2|$. Fix
$F\in\mathcal{F}_{r,d}$. We have
$$NS(F)^2=\bigg|\bigcup_{x,y\in\fq}\{(\bfs a_1,\bfs a_2)
\in(\fq^{r-1})^2:F(\bfs a_1,x) =F(\bfs a_2,y)=0\}\bigg|.$$
Then the inclusion--exclusion principle implies
\begin{align*}
\sum_{F\in \mathcal{F}_{r,d}}NS(F)^2 &=\sum_{F\in
\mathcal{F}_{r,d}}\sum_{j=1}^q\sum_{k=1}^q(-1)^{j+k}\sum_{\mathcal{X}_j\subset\fq}
\sum_{\mathcal{Y}_k\subset\fq}
\mathcal{S}(\mathcal{X}_j,\mathcal{Y}_k)\\
&=\sum_{j=1}^q\sum_{k=1}^q(-1)^{j+k} \sum_{\mathcal{X}_j\subset\fq}
\sum_{\mathcal{Y}_k\subset\fq} \sum_{F\in
\mathcal{F}_{r,d}}\mathcal{S}(\mathcal{X}_j,\mathcal{Y}_k),
\end{align*}
where $\mathcal{X}_j$ and $\mathcal{Y}_k$ run through all the
subsets of $\fq$ of cardinality $j$ and $k$, respectively, and, for
arbitrary subsets $\mathcal{X}\subset\fq$ and
$\mathcal{Y}\subset\fq$,
$$\mathcal{S}(\mathcal{X},\mathcal{Y}):=\big|\{(\bfs
a_1,\bfs a_2)\in(\fq^{r-1})^2:(\forall x\in\mathcal{X})(\forall
x\in\mathcal{Y})\,F(\bfs a_1,x)=0, F(\bfs a_2,y)=0\}\big|.$$

For $\underline{\bfs a}:=(\bfs a_1,\bfs a_2)\in(\fq^{r-1})^2$ and
subsets $\mathcal{X}\subset\fq$ and $\mathcal{Y}\subset\fq$, denote
$$\mathcal{S}_{\underline{\bfs a}}(\mathcal{X},\mathcal{Y})
:=\{F\in\mathcal{F}_{r,d}:(\forall x\in\mathcal{X})(\forall
x\in\mathcal{Y})\,F(\bfs a_1,x)=0, F(\bfs a_2,y)=0\}.$$
It follows that
\begin{align*}
\sum_{F\in \mathcal{F}_{r,d}}NS(F)^2
&=\sum_{j=1}^q\sum_{k=1}^q(-1)^{j+k}\sum_{\mathcal{X}_j\subset\fq}
\sum_{\mathcal{Y}_k\subset\fq}\sum_{\underline{\bfs
a}\in(\fq^{r-1})^2}|\mathcal{S}_{\underline{\bfs
a}}(\mathcal{X}_j,\mathcal{Y}_k)|\\
&=\sum_{\underline{\bfs a}\in(\fq^{r-1})^2}
\sum_{j=1}^q\sum_{k=1}^q(-1)^{j+k}\sum_{\mathcal{X}_j\subset\fq}
\sum_{\mathcal{Y}_k\subset\fq}|\mathcal{S}_{\underline{\bfs
a}}(\mathcal{X}_j,\mathcal{Y}_k)|=:\sum_{\underline{\bfs
a}\in(\fq^{r-1})^2}N_{\underline{\bfs a},2},
\end{align*}
where $N_{\underline{\bfs a},2}$ is defined as in \eqref{eq:
definition N_{a,2}}. If $\underline{\bfs a}\in\FF_2$, then the claim
in the proof of Proposition \ref{prop: analysis c=2 fixed} asserts
that
$$\frac{N_{\underline{\bfs a},2}}{|\mathcal{F}_{r,d}|}
=\big(P[C=1]\big)^2+\frac{q-1}{q^{2d+2}}\binom{q-1}{d}^2.$$
On the other hand, for $(\bfs a,\bfs
a)\in(\fq^{r-1})^2\setminus\FF_2$, by elementary calculations we see
that
\begin{align*}
N_{(\bfs a,\bfs a),2}&:=
\sum_{j=1}^q\sum_{k=1}^q(-1)^{j+k}\sum_{\mathcal{X}_j\subset\fq}
\sum_{\mathcal{Y}_k\subset\fq}|\mathcal{S}_{(\bfs a,\bfs
a)}(\mathcal{X}_j,\mathcal{Y}_k)|=
\sum_{j=1}^q(-1)^{j-1}\sum_{\mathcal{X}_j\subset\fq}
|\mathcal{S}_{\bfs a}(\mathcal{X}_j)|,
\end{align*}
where $\mathcal{S}_{\bfs a}(\mathcal{Z}):=
\{F\in\mathcal{F}_{r,d}:(\forall z\in\mathcal{Z})\,F(\bfs a,z)=0\}$
for any subset $\mathcal{Z}\subset\fq$. Thus,
\begin{align*}
\frac{1}{|\mathcal{F}_{r,d}|} \sum_{F\in \mathcal{F}_{r,d}}NS(F)^2
&=\sum_{\underline{\bfs a}\in\FF_2}\frac{N_{\underline{\bfs
a},2}}{|\mathcal{F}_{r,d}|}+ \frac{1}{|\mathcal{F}_{r,d}|}
\sum_{\bfs a\in \fq^{r-1}}
\sum_{j=1}^q(-1)^{j-1}\sum_{\mathcal{X}_j\subset\fq}
|\mathcal{S}_{\bfs a}(\mathcal{X}_j)|\\ &=
N_2\bigg(\big(q^{1-r}NS(r,d)\big)^2+\frac{q-1}{q^{2d+2}}\binom{q-1}{d}^2\bigg)+
\sum_{F\in \mathcal{F}_{r,d}} \frac{NS(F)}{|\mathcal{F}_{r,d}|}.
\end{align*}
The statement of the proposition follows easily from Lemma
\ref{lemma: average number vertical strips}.
\end{proof}

By the Chebyshev inequality we obtain a lower bound on the number of
$F\in \mathcal{F}_{r,d}$ for which $NS(F)$ differs a certain
proportion from the expected value $NS(r,d)$.
\begin{corollary}\label{coro: NS(F) deviating from NS(d,r)}
For $0<\alpha<1$, the number $A(\alpha)$ of $F\in \mathcal{F}_{r,d}$
for which $NS(F)\le (1-\alpha)NS(r,d)$ is bounded as
$$A(\alpha)\le \frac{1}{(\alpha\,\mu_d\,d!)^2}q^{\dim\mathcal{F}_{r,d}-1}+
\frac{1}{\alpha^2}\,\frac{1-\mu_d}{\mu_d}
q^{\dim\mathcal{F}_{r,d}-r+1}
+\mathcal{O}(q^{\dim\mathcal{F}_{r,d}-2}).$$
\end{corollary}
\begin{proof}
By Lemma \ref{lemma: average number vertical strips} and Proposition
\ref{prop: variance number of vert strips}, the Chebyshev inequality
implies
$$p_{r,d}\left(|NS(F)-NS(r,d)|\ge \alpha NS(r,d)\right)\le
\frac{NS_2(r,d)}{\alpha^2NS(r,d)^2}.$$
Taking into account that
$$\frac{NS_2(r,d)}{\alpha^2NS(r,d)^2}=
\frac{1}{(\alpha\,\mu_d\,d!)^2}q^{-1}+\frac{1-\mu_d}{\alpha^2\mu_d}q^{1-r}+\mathcal{O}(q^{-2}),$$
the corollary readily follows.
\end{proof}
%
%
\subsection{A lower bound for the entropy}
\label{subsec: bound entropy with replacement}
In order to analyze the Shannon entropy \eqref{eq: definition
entropy}, it is necessary to determine the probability $P_{\bfs
x,F}$ that an element $\bfs x:=(\bfs a,x)\in\fq^{r}$ occurs as
output on input $F\in\mathcal{F}_{r,d}$.

Given an input polynomial $F\in\mathcal{F}_{r,d}$, and the vertical
strip defined by an element $\bfs a\in\fq^{r-1}$, the SVS algorithm
proceeds to search for $\fq$--rational zeros of the univariate
polynomial $f:=\gcd\big(F(\bfs a,T),T^q-T)$. If this search is done
using the randomized algorithm of Cantor and Zassenhaus (see
\cite{CaZa81}), then all the $\fq^\times$--rational zeros of $f$ are
equiprobable (see, e.g., \cite[Section 14.3]{GaGe99}). 
The algorithm can be easily modified so that all $\fq$--rational
zeros of $f$ are equiprobable. In the sequel we shall assume that
the search of roots in $\fq$ of elements of $\mathcal{F}_{1,d}$ is
performed using a randomized algorithm for which all outputs are
equiprobable.

For the analysis of the distribution of outputs, we denote as before
by $\Omega_d:=\fq^{r_d}$ the set of all possible random choices of
elements of $\fq$ made by the routine for univariate root finding.
We consider $\Omega_d$ to be endowed with the uniform probability,
$\FF\times\mathcal{F}_{r,d}$ with the probability measure $P$ of
Section \ref{section: analysis of SVS algorithm}, and
$\FF\times\mathcal{F}_{r,d}\times \Omega_d$ with the product
probability $P\times P_{\Omega_d}$. Finally, we shall consider
probabilities related to the random variable
$C_{\mathrm{out}}:\FF\times\mathcal{F}_{r,d}\times
\Omega_d\to\fq^{r}\cup\{\emptyset\}$ defined in the following way:
for a triple $(\underline{\bfs a},F,\gamma)\in
\FF\times\mathcal{F}_{r,d}\times \Omega_d$, if $F$ has an
$\fq$--rational zero on any of the vertical strips defined by
$\underline{\bfs a}$, and $\bfs a_j$ is the first vertical strip
with this property, then $C_{\mathrm{out}}(\underline{\bfs
a},F,\gamma):=(\bfs a_j,x)$, where $x\in\fq$ is the zero of $F(\bfs
a_j,T)$ computed by the root--finding routine determined by the
random choice $\gamma$. Otherwise, we define
$C_{\mathrm{out}}(\underline{\bfs a},F,\gamma):=\emptyset$. In these
terms, the probability $P_{\bfs x,F}$ that an element $\bfs x:=(\bfs
a,x)\in\fq^{r}$ occurs as output on input $F\in\mathcal{F}_{r,d}$
may be expressed as the conditional probability $P\times
P_{\Omega_d} \big[C_{\mathrm{out}}=\bfs x|F\big]$, namely
$$P_{\bfs x,F}=P\times P_{\Omega_d}
\big[C_{\mathrm{out}}=\bfs x|F\big]:= \frac{P\times
P_{\Omega_d}\big[\{C_{\mathrm{out}}=\bfs x\}\cap
(\FF\times\{F\}\times \Omega_d)\big]}{ P\times
P_{\Omega_d}\big[\FF\times\{F\}\times \Omega_d\big]}.$$

Now we are ready to determine $P_{\bfs x,F}$. For this purpose, we
denote by $N_{\bfs a}(F)$ the number of $\fq$--rational zeros of $F$
in the vertical strip defined by $\bfs a$, i.e.,
$$N_{\bfs a}(F):=|\{x\in\fq:F(\bfs a,x)=0\}|.$$
We have the following result.
\begin{lemma}\label{lemma: prob (a,x) output nind model}
Let $F\in\mathcal{F}_{r,d}$ and $\bfs x:=(\bfs a,x)\in Z(F)$. Then
$$P_{\bfs x,F}=\frac{1}{NS(F)\,N_{\bfs a}(F)}.$$
\end{lemma}
\begin{proof}
If $\bfs x$ occurs as output at the $j$th step, then the SVS
algorithm must have chosen elements $\bfs a_1\klk \bfs a_{j-1}$ for
the first $j-1$ searches such that $N_{\bfs a_k}(F)=0$ for $1\le
k\le j-1$, and the element $\bfs a$ for the $j$th search. Finally,
the routine for finding roots of $F(\bfs a,T)$ must output $x$,
which occurs with probability $1/N_{\bfs a}(F)$.

Recall that the element $\bfs a_j\in\fq^{r-1}$ for the $j$th search
is randomly chosen among the elements of $\fq^{r-1}\setminus\{\bfs
a_1\klk \bfs a_{j-1}\}$ with equiprobability. Therefore, if $\bfs a$
arises as the choice for the $j$th step, then the SVS algorithm must
have chosen pairwise--distinct elements $\bfs a_1\klk \bfs
a_{j-1}\in\fq^{r-1}\setminus NS(F)$ for the first $j-1$ searches.
The probability of these choices is
\begin{align*}
P(N_{\bfs a_1}(F)=0\klk N_{\bfs a_{j-1}}(F)=0,\bfs a_j=\bfs a|F)& =
\prod_{k=0}^{j-2}\bigg(1-\frac{NS(F)}{q^{r-1}-k}\bigg)
\cdot\frac{1}{q^{r-1}-j+1}\\&=\frac{1}{q^{r-1}}
\frac{\binom{q^{r-1}-NS(F)}{j-1}}{\binom{q^{r-1}-1}{j-1}}.
\end{align*}
As there are $q^{r-1}-NS(F)$ elements $\bfs b\in\fq^{r-1}$ with
$N_{\bfs b}(F)=0$, the algorithm performs at most $q^{r-1}-NS(F)+1$
searches. Finally, when $\bfs a$ is chosen, the probability to find
$x$ as the $\fq$--rational zero of $F(\bfs a,T)$ is equal to
$1/N_{\bfs a}(F)$. It follows that
\begin{align*}
P_{\bfs x,F}&= \sum_{j=1}^{q^{r-1}-NS(F)+1} P(N_{\bfs a_1}(F)=0\klk
N_{\bfs a_{j-1}}(F)=0,\bfs a_j=\bfs a|F)\cdot
\frac{1}{N_{\bfs a}(F)}\\
&=\frac{1}{q^{r-1}N_{\bfs a}(F)}\sum_{j=0}^{q^{r-1}-NS(F)}
\frac{\binom{q^{r-1}-NS(F)}{j}}{\binom{q^{r-1}-1}{j}}.
\end{align*}

According to, e.g., \cite[\S 5.2, Problem 1]{GrKnPa94},
$$\sum_{j=0}^{q^{r-1}-NS(F)}
\frac{\binom{q^{r-1}-NS(F)}{j}}{\binom{q^{r-1}-1}{j}}
=\frac{q^{r-1}}{NS(F)}.$$
We conclude that
$$
P_{\bfs x,F}=\frac{1}{q^{r-1}N_{\bfs a}(F)}\, \frac{q^{r-1}}{NS(F)}
= \frac{1}{NS(F)\,N_{\bfs a}(F)}.
$$
This completes the proof of the lemma.
\end{proof}

For any $F\in\mathcal{F}_{r,d}$, consider the entropy
\begin{equation}\label{eq: entropy random}
H_F=\sum_{(\bfs a,x)\in Z(F)} \frac{\log\big(NS(F)\,N_{\bfs
a}(F)\big)}{NS(F)\,N_{\bfs a}(F)}.
\end{equation}
We aim to determine the asymptotic behavior of the average entropy
$$H:=\frac{1}{|\mathcal{F}_{r,d}|}\sum_{F\in
\mathcal{F}_{r,d}}H_F= \frac{1}{|\mathcal{F}_{r,d}|}\sum_{F\in
\mathcal{F}_{r,d}}\sum_{(\bfs a,x)\in Z(F)}
\frac{\log\big(NS(F)\,N_{\bfs a}(F)\big)}{NS(F)\,N_{\bfs a}(F)}.$$

Observe that
\begin{equation}\label{eq: number of terms entropy}
\sum_{F\in \mathcal{F}_{r,d}}\sum_{(\bfs a,x_r)\in Z(F)}1=
\sum_{(\bfs a,x)\in\fq^r}|\{F\in \mathcal{F}_{r,d}:F(\bfs a,x)=0\}|=
q^{\dim\mathcal{F}_{r,d}+r-1}
\end{equation}
Further, the function $h:(0,+\infty)\to\R$, $h(x):=\log x/x$ is
increasing in the interval $[e,+\infty)$ and convex in the interval
$[e^{3/2},+\infty)$. By Corollary \ref{coro: NS(F) deviating from
NS(d,r)}, the probability of the set of $F\in\mathcal{F}_{r,d}$
having up to $e^{3/2}=4.48\dots$ vertical strips is
$\mathcal{O}(q^{-1})$. Therefore,
\begin{align}
H&= \frac{\ds\sum_{F\in \mathcal{F}_{r,d}}\sum_{(\bfs a,x)\in
Z(F)}1}{|\mathcal{F}_{r,d}|}\,\frac{{\ds\sum_{F\in
\mathcal{F}_{r,d}}\sum_{(\bfs a,x)\in Z(F)}} \frac{
\log(NS(F)\,N_{\bfs a}(F))}{NS(F)\,N_{\bfs a}(F)}}{ {\ds\sum_{F\in
\mathcal{F}_{r,d}}\sum_{(\bfs a,x)\in Z(F)}}1}
\nonumber\\
&\ge q^{r-1}\,h\left(\frac{{\ds\sum_{F\in
\mathcal{F}_{r,d}}\sum_{(\bfs a,x)\in Z(F)} }NS(F)\,N_{\bfs
a}(F)}{{\ds\sum_{F\in \mathcal{F}_{r,d}}\sum_{(\bfs a,x)\in
Z(F)}}1}\right)(1+\mathcal{O}(q^{-1})).\label{eq: ineq entropy
convex function}
\end{align}

Next we analyze the numerator
$$\mathcal{N}:=\sum_{F\in
\mathcal{F}_{r,d}}\sum_{(\bfs a,x)\in Z(F)} NS(F)\,N_{\bfs a}(F)$$
in the argument of $h$ in the last expression.
\begin{lemma}\label{lemma: sum for entropy}
We have $\mathcal{N}=2\,\mu_d\,q^{2r-2+
\dim\mathcal{F}_{r,d}}(1+\mathcal{O}(q^{-1}))$.
\end{lemma}
\begin{proof}
For $F\in\mathcal{F}_{r,d}$ and $\bfs a\in VS(F)$, we have
$$
NS(F)=\bigg|\bigcup_{x\in\fq}\{\bfs a\in\fq^{r-1}:F(\bfs
a,x)=0\}\bigg|,\ N_{\bfs a}(F)=\left|\{x\in\fq:F(\bfs
a,x)=0\}\right|.$$
As a consequence,
\begin{align*}
\mathcal{N}&=\sum_{F\in \mathcal{F}_{r,d}}
\sum_{\stackrel{\scriptstyle (\bfs a,x)\in\fq^r}{ {F(\bfs a,x)=0}}}
\sum_{\stackrel{\scriptstyle y\in\fq}{F(\bfs a,y)=0}}
\bigg|\bigcup_{z\in\fq}\{\bfs b\in\fq^{r-1}:F(\bfs b,z)=0\}\bigg|\\
&=\sum_{F\in \mathcal{F}_{r,d}}
\sum_{\stackrel{\scriptstyle (\bfs a,x)\in\fq^r}{ {F(\bfs a,x)=0}}}
\sum_{\stackrel{\scriptstyle y\in\fq}{F(\bfs a,y)=0}}
\sum_{k=1}^q(-1)^{k-1}\sum_{\stackrel{\scriptstyle\mathcal{Z}_k\subset\fq}{|\mathcal{Z}_k|=k}}
\big|\{\bfs b\in\fq^{r-1}:F(\bfs
b,T)|_{\mathcal{Z}_k}\equiv0\}\big|\\
&=\sum_{k=1}^q(-1)^{k-1}
\sum_{\bfs a\in\fq^{r-1}}
\sum_{x\in\fq} \sum_{y\in\fq}
\sum_{\stackrel{\scriptstyle\mathcal{Z}_k\subset\fq}{|\mathcal{Z}_k|=k}}
\mathcal{N}_{\bfs a,x,y,\mathcal{Z}_k},
\end{align*}
where
\begin{align*}
\mathcal{N}_{\bfs a,x,y,\mathcal{Z}_k}&:=
\sum_{\stackrel{\scriptstyle F\in \mathcal{F}_{r,d}}{F(\bfs
a,x)=F(\bfs a,y)=0}}\big|\{\bfs b\in\fq^{r-1}:F(\bfs
b,T)|_{\mathcal{Z}_k}\equiv0\}\big|\\
&=\sum_{\bfs b\in\fq^{r-1}}\big|\{F\in\mathcal{F}_{r,d}:F(\bfs
a,x)=0,F(\bfs a,y)=0,F(\bfs b,T)|_{\mathcal{Z}_k}\equiv0\}\big|.
\end{align*}

Suppose that $k\le d$. For $\bfs b\not=\bfs a$ and $x\not= y$, the
equalities $F(\bfs a,x)=0,F(\bfs a,y)=0,F(\bfs
b,T)|_{\mathcal{Z}_k}\equiv0$ are linearly--independent conditions
on the coefficients of $F$. If $\bfs b\not=\bfs a$ and $x=y$, then
we have $k+1$ linearly--independent conditions. Finally, for $\bfs
b=\bfs a$, the number of linearly--independent conditions depends on
the size of the intersection $\{x,y\}\cap\mathcal{Z}_k$. It follows
that
$$\mathcal{N}_{\bfs a,x,y,\mathcal{Z}_k}=(q^{r-1}-1)\,
q^{\dim\mathcal{F}_{r,d}-k-|\{x,y\}|}+
q^{\dim\mathcal{F}_{r,d}-\min\{d+1,|\{x,y\}\cup \mathcal{Z}_k|\}}.$$
Therefore, by elementary calculations we obtain
\begin{align*}
\sum_{x\in\fq} \sum_{y\in\fq}
\sum_{\stackrel{\scriptstyle\mathcal{Z}_k\subset\fq}{|\mathcal{Z}_k|=k}}
\mathcal{N}_{\bfs a,x,y,\mathcal{Z}_k}&= (q^{r-1}-1)\binom{q}{k}
q^{\dim\mathcal{F}_{r,d}-k}\bigg(\frac{q^2-q}{q^2}+\frac{q}{q}\bigg)(1+\mathcal{O}(q^{1-r}))
\\
&=\frac{2q-1}{q} (q^{r-1}-1)\binom{q}{k}
q^{\dim\mathcal{F}_{r,d}-k}(1+\mathcal{O}(q^{1-r})).
\end{align*}

Now assume that $k>d$. Then the condition $F(\bfs
b,T)|_{\mathcal{Z}_k}\equiv0$ is equivalent to $F(\bfs b,T)=0$.
Arguing as above, we deduce that
$$\sum_{x\in\fq} \sum_{y\in\fq}
\sum_{\stackrel{\scriptstyle\mathcal{Z}_k\subset\fq}{|\mathcal{Z}_k|=k}}
\mathcal{N}_{\bfs a,x,y,\mathcal{Z}_k}=\frac{2q-1}{q}
(q^{r-1}-1)\binom{q}{k}
q^{\dim\mathcal{F}_{r,d}-(d+1)}(1+\mathcal{O}(q^{1-r})).$$

Putting these equalities together and using \eqref{eq: identities
combinatorial numbers}, we obtain
\begin{align*}
\mathcal{N}=&2q^{2r-2+\dim\mathcal{F}_{r,d}}\frac{2q-1}{2q}(1-q^{1-r})\\
&\bigg(\sum_{k=1}^d(-1)^{k-1}\binom{q}{k}q^{-k}
+\sum_{k=d+1}^q(-1)^{k-1}\binom{q}{k}q^{-d-1}\bigg)(1+\mathcal{O}(q^{1-r}))\\
=&2\,\mu_d\,q^{2r-2+\dim\mathcal{F}_{r,d}}(1+\mathcal{O}(q^{-1})).
\end{align*}
This finishes the proof of the lemma.
\end{proof}

Combining \eqref{eq: ineq entropy convex function} with \eqref{eq:
number of terms entropy} and Lemma \ref{lemma: sum for entropy}, it
follows that
$$H\ge q^{r-1}h\left(\frac{2\,\mu_d\,q^{2r-2+\dim\mathcal{F}_{r,d}}
(1+\mathcal{O}(q^{-1}))}{
q^{r-1+\dim\mathcal{F}_{r,d}}}\right)(1+\mathcal{O}(q^{-1})).$$
In other words, we have the following result.
\begin{theorem}\label{th: entropy}
If $H$ denotes the average entropy of the SVS algorithm, then
$$
H\ge \frac{1}{2\mu_d}\log(q^{r-1}) (1+\mathcal{O}(q^{-1})).
$$
\end{theorem}

Recall that, according to \eqref{eq: bound entropy ideal algorithm},
for an algorithm for which the outputs are equidistributed we have
the upper bound $H\le \log(q^{r-1})$. For large $d$ we have
$$\frac{1}{2\mu_d}\approx \frac{1}{2(1-e^{-1})}\approx 0.79.$$
We may therefore paraphrase Theorem \ref{th: entropy} as saying that
the SVS algorithm is at least $79$ per cent as good as any ``ideal''
algorithm.
%
%
\section{Simulations on test examples}
\label{section: simulations}
We end the paper with a description of the results on the number of
searches that were obtained by executing the SVS algorithm on random
samples of elements $\mathcal{F}_{r,d}$, for given values of $q$,
$r$ and $d$. Recall that $C:\FF\times \mathcal{F}_{r,d}\mapsto
\N\cup\{\infty\}$ denotes the random variable which counts the
number of searches that are performed for all possible choices of
vertical strips. Theorem \ref{th: prob C=s for nind} shows that
$$P[C=s]\approx (1-\mu_d)^{s-1}\mu_d.$$
The simulations we exhibit were aimed to test whether the
right--hand side of the previous expression approximates the
left--hand side on the examples considered. For a random sample
$\mathcal{S}\subset \mathcal{F}_{r,d}$ and $\underline{\bfs
a}\in\FF_s$, we use the following notations:
$$p_{\underline{\bfs a}}:=p_{r,d}[\mathcal{S}\cap
C_{\underline{\bfs a}}=s],\quad
\widehat{p}_s:=(1-\mu_d)^{s-1}\mu_d.$$
We take $N:=30$ choices of $\underline{\bfs a}\in \FF_s$, and
compute the sample mean
$$\overline{p}_s:=\sum_{i=1}^N \frac{p_{\underline{\bfs
a}_i}}{N}.$$
Furthermore, we consider the corresponding relative errors:
$$\epsilon_{s}:=\frac{|\overline{p}_s-\widehat{p}_s|}{\widehat{p}_s}.$$
Finally, we compare the average number $\overline{N}{}^{\,q}_{r,d}$
of vertical strips searched with its theoretical upper bound
according to Theorem \ref{th: average-case compl r>2}, namely
$1/\mu_d$.

We consider only relatively moderate values of $s$, since for higher
values the probability $p_{\underline{\bfs a}}$ is so small that the
corresponding information becomes uninteresting. This also explains
the fact that relative errors $\epsilon_s$ tend to grow as $s$
grows. Finally, we remark that, although polynomials without
$\fq$--rational zeros occur in some of the experiments described
below, the number of such polynomial is so small that it does not
affect the average behavior of our simulations.
%
%
\subsection{Examples with $r:=2$ and $q:=67$ and $q:=8$}
In this section we consider random samples of bivariate polynomials
with coefficients in the finite field $\F_{\!67}$. In Table
\ref{table: q=67, r=2, d=30} we consider a random sample
$\mathcal{S}$ of $1000000$ polynomials of $\F_{\!67}[X_1,X_2]$ of
degree at most $d:=30$ and analyze how many vertical strips are
searched on this sample. Therefore,
%
%
%
we have $\widehat{p}_s:=(1-\mu_{30})^{s-1}\mu_{30}$, where
$\mu_{30}:=0.6321205588\dots$. Further, we have
$\overline{N}{}^{\,67}_{2,30}=1.574924\dots$, to be compared with
$1/\mu_{30}=1.581977\dots$.

\begin{center}
\begin{table}[h]
\caption{Random sample with $q=67$, $r=2$ and $d=30$.}
\label{table: q=67, r=2, d=30}
\begin{tabular}{|c|c|c|c|c|c|}
     \hline
    $s$  &  $\overline{p}_s$ & $\widehat{p}_s$ & $\epsilon_s$ \\
    \hline
    $1$  &  $0.635031$ & $0.632121$ &  $0.004583$ \\
    \hline
    $2$  &  $0.231664$ & $0.232544$ &  $0.003799$ \\
    \hline
    $3$  & $0.084627$  & $0.085548$ &  $0.010889$ \\
    \hline
    $4$  & $0.030921$ & $0.031471$  &  $0.017789$ \\
    \hline
    $5$  & $0.011279$ & $0.011578$  &  $0.026473$ \\
    \hline
    $6$  & $0.004101$ & $0.004259$  &  $0.038575$ \\
    \hline
    $7$  & $0.001509$ & $0.001567$  &  $0.038166$ \\
    \hline
    $8$  & $0.000553$ & $0.000576$  &  $0.042349$ \\
    \hline
    $9$  &  $0.000199$ & $0.000212$ &  $0.067918$ \\
    \hline
    $10$ & $0.000076$ & $0.000078$  &  $0.030513$ \\
    \hline
    $11$ & $0.000025$ & $0.000029$  &  $0.161872$ \\
    \hline
    $12$ & $0.000010$ & $0.000011$  &  $0.038441$ \\
    \hline
    $13$ & $0.000038$ & $0.000003$  &  $0.022074$ \\
    \hline
    $14$ & $0.000011$ & $0.000001$  &  $0.339501$ \\
    \hline
    $15$ & $0.000001$  & $0.000001$ &  $0.051253$ \\
    \hline
\end{tabular}
\end{table}
\end{center}

Our second example concerns a sample $1000000$ polynomials of
$\F_{\!67}[X_1,X_2]$ of degree at most $d:=5$. We have
$\widehat{p}_s:=(1-\mu_{5})^{s-1}\mu_{5} $, where
$\mu_{5}:=0.6333333\dots$. The corresponding results are summarized
in Table \ref{table: q=67, r=2, d=5}. We observe that
$\overline{N}{}^{\,67}_{2,5}=1.572816\dots$, to be compared with
$1/\mu_5=1.578947\dots$.

\begin{center}
\begin{table}[h]
\caption{Random sample with $q=67$, $r=2$ and $d=5$.}
\label{table: q=67, r=2, d=5}
\begin{tabular}{|c|c|c|c|c|c|c|c|}
      \hline
    $s$  & $\overline{p}_s$ & $\widehat{p}_s$ &  $\epsilon_s$ \\
\hline
    $1$ & $0.635885$ & $0.633333$ & $0.004012$ \\
    \hline
    $2$ & $0.231459$ & $0.232222$ & $0.003298$ \\
    \hline
    $3$ & $0.084318$ & $0.085148$ & $0.009844$ \\
    \hline
    $4$ & $0.030727$ & $0.031221$ & $0.016085$ \\
    \hline
    $5$ & $0.011188$ & $0.011448$ & $0.023224$ \\
    \hline
    $6$ & $0.004091$ & $0.004197$ & $0.025996$ \\
    \hline
    $7$ & $0.001481$ & $0.001539$ & $0.039029$ \\
    \hline
    $8$ & $0.000543$ & $0.000564$ & $0.040109$ \\
    \hline
    $9$ & $0.000195$ & $0.000207$ & $0.056976$ \\
    \hline
    $10$& $0.000069$ & $0.000076$ & $0.085938$ \\
    \hline
    $11$& $0.000029$ & $0.000028$ & $0.030685$ \\
    \hline
    $12$& $0.000009$ & $0.000010$ & $0.129198$ \\
    \hline
    $13$& $0.000003$ & $0.000003$ & $0.133380$ \\
    \hline
    $14$& $0.000002$ & $0.000001$ & $0.085740$ \\
    \hline
    $15$& $0.000001$ & $0.000001$ & $0.057169$ \\
    \hline
\end{tabular}
\end{table}
\end{center}

We end this section by considering polynomials with coefficients in
a non--prime field, namely $\F_{\!8}[X_1,X_2]$. In this case,
$\widehat{p}_s:=(1-\mu_3)^{s-1}\mu_3$, where $\mu_3:=0.666666\dots$.
In Table \ref{table: q=8, r=2, d=3} the results for a sample of
$100000$ polynomials of degree at most $d:=3$ are exhibited. We have
$\overline{N}{}^{\,8}_{3,3}=1.504512\dots$, to be compared with
$1/\mu_3=1.5$.
\begin{center}
\begin{table}[h]
\caption{Random sample with $q=8$, $r=3$ and $d=3$.}
\label{table: q=8, r=2, d=3}
\begin{tabular}{|c|c|c|c|c| c| c|}
      \hline
 $s$ & $\overline{p}_s$ & $\widehat{p}_s$ & $\epsilon_s$ \\
\hline
    $1$  &  $0.663161$ & $0.666666$ & $0.005259$ \\
    \hline
    $2$  &  $0.222801$ & $0.222222$ & $0.002605$ \\
    \hline
    $3$  &  $0.075617$ & $0.074074$ & $0.014151$ \\
    \hline
    $4$  &  $0.025319$ & $0.024691$ & $0.020831$ \\
    \hline
    $5$  &  $0.008725$ & $0.008230$ & $0.060146$ \\
    \hline
    $6$  &  $0.002859$ & $0.002743$ & $0.042289$ \\
    \hline
\end{tabular}
\end{table}
\end{center}
%
%
\subsection{Examples with $r:=3$ and $q:=11$ and $q:=67$}
Finally, we consider two samples of $1000000$ polynomials of
$\fq[X_1,X_2,X_3]$. The first sample contains polynomials of degree
at most $d:=5$ with coefficients in $\F_{\!11}$, while the second
one contains polynomials of degree at most $d:=5$ with coefficients
in $\F_{\!67}$. Results are exhibited in Tables \ref{table: q=11,
r=3, d=5} and \ref{table: q=67, r=3, d=5} respectively. The average
numbers of searched vertical strips are
$\overline{N}{}^{\,11}_{3,5}=1.539646\dots$ and
$\overline{N}{}^{\,67}_{3,5}=1.572975\dots$, both to be compared
with $1/\mu_{5}=1.578947\dots$.
\begin{center}
\begin{table}[h]
\caption{Random sample with $q=11$, $r=3$ and $d=5$.}
\label{table: q=11, r=3, d=5}
\begin{tabular}{|c|c|c|c|c|c|}
      \hline
 $s$ & $\overline{p}_s$ & $\widehat{p}_s$ & $\epsilon_s$ \\
\hline
    $1$ & $0.649494$ & $0.633333$  & $0.024881$ \\
    \hline
    $2$ & $0.227637$ & $0.232222$  & $0.020145$ \\
    \hline
    $3$ & $0.079769$ & $0.085148$  & $0.067430$ \\
    \hline
    $4$ & $0.027999$ & $0.031221$  & $0.115075$ \\
    \hline
    $5$ & $0.009822$ & $0.011448$  & $0.165519$ \\
    \hline
    $6$ & $0.003419$ & $0.004198$  & $0.227683$ \\
    \hline
    $7$ & $0.001213$ & $0.001539$  & $0.269344$ \\
    \hline
    $8$ & $0.000421$ & $0.000564$  & $0.340555$ \\
    \hline
    $9$ & $0.000149$ & $0.000207$  & $0.382851$ \\
    \hline
   $10$ & $0.000050$ & $0.000076$  & $0.504379$ \\
    \hline
   $11$ & $0.000017$ & $0.000028$  & $0.662509$ \\
    \hline
   $12$ & $0.000002$ & $0.000010$  & $0.500062$ \\
    \hline
   $13$ & $0.000002$ & $0.000004$  & $0.726225$ \\
    \hline
   $14$ & $0.000001$ & $0.000001$  & $0.523767$ \\
    \hline
   $15$ & $0.000000$ & $0.000001$  & $2.017058$ \\
    \hline
\end{tabular}
\end{table}
\end{center}
\begin{center}
\begin{table}[h]
\caption{Random sample with $q=67$, $r=3$ and $d=5$.}
\label{table: q=67, r=3, d=5}
\begin{tabular}{|c|c|c|c|c| c| c|}
      \hline
 $s$ & $\overline{p}_s$ & $\widehat{p}_s$ & $\epsilon_s$ \\
\hline
    $1$  & $0.635802$ & $0.633333$ & $0.003883$ \\
    \hline
    $2$  & $0.231571$ & $0.232222$ & $0.002810$ \\
    \hline
    $3$  & $0.084285$ & $0.085148$ & $0.010237$ \\
    \hline
    $4$  & $0.030732$ & $0.031221$ & $0.015898$ \\
    \hline
    $5$  & $0.011192$ & $0.011447$ & $0.022809$ \\
    \hline
    $6$  & $0.004081$ & $0.004197$ & $0.028645$ \\
    \hline
    $7$  & $0.001482$ & $0.001539$ & $0.038865$ \\
    \hline
    $8$  & $0.000541$ & $0.000564$ & $0.042865$ \\
    \hline
    $9$  & $0.000199$ & $0.000207$ & $0.039628$ \\
    \hline
    $10$ & $0.000071$ & $0.000076$ & $0.062618$ \\
    \hline
    $11$ & $0.000027$ & $0.000028$ & $0.017780$ \\
    \hline
    $12$ & $0.000010$ & $0.000010$ & $0.003320$ \\
    \hline
    $13$ & $0.000003$ & $0.000004$ & $0.078891$ \\
    \hline
    $14$ & $0.000001$ & $0.000001$ & $0.111938$ \\
    \hline
    $15$ & $0.000000$ & $0.000001$ & $0.257107$ \\
    \hline
\end{tabular}
\end{table}
\end{center}

Summarizing, the results of Tables \ref{table: q=67, r=2,
d=30}--\ref{table: q=67, r=3, d=5} show that the behavior predicted
by the asymptotic estimates of Theorems \ref{th: prob C=s for nind}
and \ref{th: average-case compl r>2} is also appreciated in the
numerical experiments we perform. Nevertheless, as the cost of the
SVS algorithm grows exponentially with the number $r$ of variables
under consideration, our experiments only considered the cases $r=2$
and $r=3$.
\section*{Acknowledgements}
The authors gratefully acknowledge the comments by the anonymous
referees, which helped to significantly improve the presentation of
the results of this paper.
%
%
\providecommand{\bysame}{\leavevmode\hbox
to3em{\hrulefill}\thinspace}
\providecommand{\MR}{\relax\ifhmode\unskip\space\fi MR }
\providecommand{\MRhref}[2]{%
  \href{http://www.ams.org/mathscinet-getitem?mr=#1}{#2}
} \providecommand{\href}[2]{#2}


\begin{thebibliography}{10}

\bibitem{Bach91}
E.~Bach, \emph{Realistic analysis of some randomized algorithms}, J.
Comput.
  System Sci. \textbf{42} (1991), 30--53.

\bibitem{BeLe12}
C.~{Beltr\'an} and A.~Leykin, \emph{Certified numerical homotopy
tracking},
  Exp. Math. \textbf{21} (2012), no.~1, 69--83.

\bibitem{BePa11}
C.~{Beltr\'an} and L.M. Pardo, \emph{Fast linear homotopy to find
approximate
  zeros of polynomial systems}, Found. Comput. Math. \textbf{11} (2011),
  95--129.

\bibitem{BiSD59}
B.~Birch and H.~{Swinnerton-Dyer}, \emph{Note on a problem of
{Chowla}}, Acta
  Arith. \textbf{5} (1959), no.~4, 417--423.

\bibitem{CaMa06a}
A.~Cafure and G.~Matera, \emph{Fast computation of a rational point
of a
  variety over a finite field}, Math. Comp. \textbf{75} (2006), no.~256,
  2049--2085.

\bibitem{CaMa06}
\bysame, \emph{Improved explicit estimates on the number of
solutions of
  equations over a finite field}, Finite Fields Appl. \textbf{12} (2006),
  no.~2, 155--185.

\bibitem{CaZa81}
D.G. Cantor and H.~Zassenhaus, \emph{A new algorithm for factoring
polynomials
  over finite fields}, Math. Comp. \textbf{36} (1981), 587--592.

\bibitem{CeMaPePr14}
E.~Cesaratto, G.~Matera, M.~{P\'erez}, and M.~Privitelli, \emph{On
the value
  set of small families of polynomials over a finite field, {I}}, J. Combin.
  Theory Ser. A \textbf{124} (2014), no.~4, 203--227.

\bibitem{Cohen73}
S.~Cohen, \emph{The values of a polynomial over a finite field},
Glasg. Math.
  J. \textbf{14} (1973), no.~2, 205--208.

\bibitem{DaTa09}
C.~{D'Andrea} and L.~{Tabera}, \emph{Tropicalization and
irreducibility of
  generalized {Vandermonde} determinants}, Proc. Amer. Math. Soc. \textbf{137}
  (2009), no.~11, 3647--3656.

\bibitem{Feller68}
W.~Feller, \emph{An introduction to probability theory and its
applications.
  {Vol. I}}, 3rd ed., John Wiley \& Sons, Inc., New York, 1968.

\bibitem{Feller71}
\bysame, \emph{An introduction to probability theory and its
applications.
  {Vol. II}}, 2nd ed., John Wiley \& Sons, Inc., New York, 1971.

\bibitem{FlSe08}
P.~Flajolet and R.~Sedgewick, \emph{Analytic combinatorics},
Cambridge Univ.
  Press, Cambridge, 2008.

\bibitem{GaGe99}
J.~von~zur {Gathen} and J.~Gerhard, \emph{Modern computer algebra},
Cambridge
  Univ. Press, Cambridge, 1999.

\bibitem{GaShSi03}
J.~von~zur Gathen, I.~Shparlinski, and A.~Sinclair, \emph{Finding
points on
  curves over finite fields}, SIAM J. Comput. \textbf{32} (2003), no.~6,
  1436--1448.

\bibitem{GaViZi13}
J.~von~zur {Gathen}, A.~Viola, and K.~Ziegler, \emph{Counting
reducible,
  powerful, and relatively irreducible multivariate polynomials over finite
  fields}, SIAM J. Discrete Math. \textbf{27} (2013), no.~2, 855--891.

\bibitem{GrKnPa94}
R.~{Graham}, D.~{Knuth}, and O.~{Patashnik}, \emph{{Concrete
mathematics: a
  foundation for computer science}}, 2nd ed., Addison--Wesley, Reading,
  Massachusetts, 1994.

\bibitem{Knkn90}
A.~Knopfmacher and J.~Knopfmacher, \emph{Counting polynomials with a
given
  number of zeros in a finite field}, Linear Multilinear Algebra \textbf{26}
  (1990), no.~4, 287--292.

\bibitem{LiNi83}
R.~Lidl and H.~Niederreiter, \emph{Finite fields}, Addison--Wesley,
Reading,
  Massachusetts, 1983.

\bibitem{Matera10}
G.~Matera, \emph{The computation of rational solutions of polinomial
systems
  over a finite field}, Libro de actas de las VII Jornadas de Matemática
  Discreta y Algorítmica (Santander, Spain) (D.~Sadornil et~al., ed.), 2010,
  pp.~9--33.

\bibitem{MaPePr14}
G.~Matera, M.~{P\'erez}, and M.~Privitelli, \emph{On the value set
of small
  families of polynomials over a finite field, {II}}, Acta Arith. \textbf{165}
  (2014), no.~2, 141--179.

\bibitem{MaPePr16}
\bysame, \emph{On the value set of small
  families of polynomials over a finite field, {III}},
Contemporary Developments in Finite Fields and Applications
(A.~Canteaut et~al., ed.), 2016, World Sci. Publ., Hackensack, NJ,
  pp.~217--243.

\bibitem{MuPa13}
G.~Mullen and D.~Panario, \emph{Handbook of finite fields}, CRC
Press, Boca
  Raton, FL, 2013.

\end{thebibliography}
%
\end{document}